\documentclass{article}
\usepackage{amsmath,amsthm,mathtools,mathrsfs,amsfonts,amssymb,stmaryrd,xy}
\usepackage[a4paper]{geometry}
\usepackage[pdftex]{hyperref}
\theoremstyle{definition}
\newtheorem{remark}{Remark}[section]

\newtheorem{example}[remark]{Example}
 \theoremstyle{plain}
\newtheorem{definition}[remark]{Definition}
\newtheorem{theorem}[remark]{Theorem}
\newtheorem{proposition}[remark]{Proposition}
\newtheorem{corollary}[remark]{Corollary}
\newtheorem{lemma}[remark]{Lemma}

\newtheorem{question}[remark]{Question}
\newtagform{empty}{}{}

\DeclareMathOperator{\Id}{id}

\DeclareMathOperator{\Ad}{Ad}

\DeclareMathOperator{\Ind}{Ind}

\DeclareMathOperator{\Mor}{Mor}

\newcommand{\spacestyle}{\scriptstyle}
\newcommand{\mapstyle}{\textstyle}
\newcommand{\smalldiagram}{\def\labelstyle{\mapstyle}}%\def\objectstyle{\spacestyle}}

\newcommand{\hDelta}{\widehat{\Delta}}

\newcommand{\lnspan}{\big[} \newcommand{\rnspan}{\big]}

\newcommand{\frakB}{\mathfrak{B}}
\newcommand{\frakC}{\mathfrak{C}}
\newcommand{\frakD}{\mathfrak{D}}
\newcommand{\frakH}{\mathfrak{H}}
\newcommand{\frakK}{\mathfrak{K}}

\newcommand{\cbasel}[2]{(\mathfrak{#1}, \mathfrak{#2}, \mathfrak{#1}^{\dagger})}
\newcommand{\cbases}[2]{{_{\mathfrak{#1}}}\mathfrak{#2}_{\mathfrak{#1}^{\dagger\!}}} 
\newcommand{\cbaseos}[2]{{_{\mathfrak{#1}^{\dagger\!}}}\mathfrak{#2}_{\mathfrak{#1}}} 
\newcommand{\cbasesb}{\cbases{B}{H}}
\newcommand{\cbasesc}{\cbases{C}{K}}
\newcommand{\cbaseosc}{\cbaseos{C}{K}}
\newcommand{\cbaseosb}{\cbaseos{B}{H}}

\newcommand{\lt}{\triangleleft}
\newcommand{\rt}{\triangleright}

\newcommand{\hbeta}{\widehat{\beta}}
\newcommand{\hA}{\widehat{A}} \newcommand{\ha}{\widehat{a}}

\newcommand{\frei}{\,\cdot\,}
\newcommand{\mycong}{\xrightarrow{\cong}}

\newcommand{\rtensor}[3]{ {_{#1}\! \underset{#2}{\otimes}\! {}_{#3}}}
\newcommand{\vtensor}[3]{ {_{#1} \underset{#2}{\bar{\otimes}} {}_{#3}}}
\newcommand{\htensor}[2]{\rtensor{#1}{\frakH}{#2}}

\newcommand{\rtensorab}{\htensor{\alpha}{\beta}}
\newcommand{\rtensorh}{\underset{\frakH}{\otimes}}

\newcommand{\fibre}[3]{ {_{#1}\! \underset{#2}{\ast}\! {}_{#3}}}

\newcommand{\fsource}{\htensor{\hbeta}{\alpha}}
\newcommand{\frange}{\rtensorab}

\newcommand{\sfsource}{\fsource}%{\stensor{\hbeta}{\alpha}}
\newcommand{\sfrange}{\frange}%{\stensor{\alpha}{\beta}}

\newcommand{\tl}{\ensuremath \olessthan}
\newcommand{\tr}{\ensuremath \ogreaterthan}

\newcommand{\Hsource}{H \fsource H}
\newcommand{\Hrange}{H \frange H}

\newcommand{\sHsource}{H \sfsource H}
\newcommand{\sHrange}{H \sfrange H}

\newcommand{\HfibreK}{H \rtensorab K}
\newcommand{\LfibreM}{L \rtensor{\gamma}{\frakH}{\delta} M}
\newcommand{\vHfibreK}{H {_{\tilde
      \rho_{\alpha}}}\bar{\underset{\mu}{\otimes}} {}_{\tilde
    \rho_{\beta}} K} 

\newcommand{\Hone}{H \sfsource H \sfsource H}

\newcommand{\Htwo}{H \sfrange H \sfsource  H}

\newcommand{\Hthree}{H \sfrange H \sfrange H}

\newcommand{\Hfour}{(\sHsource) \rtensor{\alpha \lt \alpha}{\frakH}{\beta} H}

\newcommand{\Hfive}{H \rtensor{\hbeta}{\frakH}{\alpha \rt \alpha} (\sHrange)}

\newcommand{\Hfourlt}{H \sfsource H \rtensor{\beta}{\frakH}{\alpha} H}
\newcommand{\Hfourrt}{\big(\sHrange\big) \rtensor{\hbeta \lt
    \beta}{\frakH}{\alpha} H}
\newcommand{\fibreab}{\fibre{\alpha}{\frakH}{\beta}}
\newcommand{\AfibreA}{A \fibreab A}

\newcommand{\AfibreB}{A \fibreab B}

\newcommand{\ind}[2]{\mathrm{Ind}_{#1}({#2})}
\newcommand{\indgA}{\ind{\Gamma}{A}}

\newcommand{\kalpha}[1]{|\alpha\rangle_{\leg{#1}}}
\newcommand{\balpha}[1]{\langle\alpha|_{\leg{#1}}}
\newcommand{\kbeta}[1]{|\beta{}\rangle_{\leg{#1}}}
\newcommand{\bbeta}[1]{\langle\beta|_{\leg{#1}}}
\newcommand{\khbeta}[1]{|\hbeta{}\rangle_{\leg{#1}}}
\newcommand{\bhbeta}[1]{\langle\hbeta|_{\leg{#1}}}

\newcommand{\Nmu}{\ensuremath \mathfrak{N}_{\mu}}

\newcommand{\cfact}{\ensuremath \mathrm{C}^{*}\mathrm{\text{-}fact}}

\newcommand{\Vl}[1]{V_{\leg{#1}}}

\newcommand{\vnast}{\bar{\ast}}
\newcommand{\xright}[1]{\xrightarrow{ #1}}

\newcommand{\leg}[1]{[#1]}
\title{$C^{*}$-pseudo-multiplicative unitaries}

\author{Thomas Timmermann\\[1ex]
  \texttt{timmermt@math.uni-muenster.de}\\ SFB 478
  ``Geometrische Strukturen in der Mathematik''\\ Hittorfstr.\
  27, 48149 M\"unster}

\date{\today}

\begin{document} \xyrequire{matrix} \xyrequire{arrow}
%\CompileMatrices

\maketitle

\abstract{We introduce $C^{*}$-pseudo-multiplicative unitaries and
  (concrete) Hopf $C^{*}$-bimodules, which are $C^{*}$-algebraic
  variants of the pseudo-multiplicative unitaries on Hilbert spaces
  and the Hopf-von Neumann-bimodules studied by Enock, Lesieur, and
  Vallin \cite{enock:1,enock:2,enock:10,lesieur,vallin:1,vallin:2}. Moreover,
  we associate to every regular $C^{*}$-pseudo-multiplicative unitary
  two Hopf-$C^{*}$-bimodules and discuss examples related to locally
  compact groupoids.  }

\section{Introduction}

Multiplicative unitaries, introduced by Baaj and Skandalis
\cite{baaj:2}, play a central r\^ole in operator-algebraic approaches
to quantum groups \cite{vaes:1,vaes:30,masuda}. Most importantly, one
can associate to every locally compact quantum group a manageable
multiplicative unitary, and to every manageable multiplicative unitary
two Hopf $C^{*}$-algebras or Hopf von Neumann-algebras called the
``legs'' of the unitary. One of these legs coincides with the initial
quantum group, and the other one is its generalized Pontrjagin dual.

In this article, we introduce $C^{*}$-pseudo-multiplicative unitaries
and Hopf $C^{*}$-bimodules, and associate to every regular
$C^{*}$-pseudo-multiplicative unitary two Hopf $C^{*}$-bimodules,
generalizing the construction of Baaj and Skandalis \cite{baaj:2}. We
think that these two concepts can form the starting point for the
development of a theory of locally compact quantum groupoids in the
setting of $C^{*}$-algebras.

In the setting of von Neumann algebras, a satisfactory theory of
locally compact quantum groupoids has already been developed by
Lesieur \cite{lesieur}, building on the concepts of a
pseudo-multiplicative unitary and of a Hopf von Neumann bimodule
introduced by Vallin \cite{vallin:1,vallin:2}.  Our
$C^{*}$-pseudo-multiplicative unitaries and Hopf $C^{*}$-bimodules
turn out to be closely related to their von Neumann algebraic
counterparts.

Let us mention that another approach to the problems pursued in this
article was developed in the PhD thesis of the author
\cite{timmermann:thesis}. The approach presented here allows us to
drop a rather restrictive condition (decomposability) needed in
\cite{timmermann:thesis}, and to work in the framework of
$C^{*}$-algebras instead of the somewhat exotic $C^{*}$-families.  A
comparison between the two approaches is in preparation.

This work was supported by the SFB 478 ``Geometrische Strukturen in
der Mathematik''\footnote{funded by the Deutsche
  Forschungsgemeinschaft (DFG)} and partially pursued during a
stay at the ``Special Programme on Operator Algebras'' at the Fields
Institute in Toronto, Canada.

\paragraph{Organization}
This article is organized as follows:

First, we fix  notation and terminology, and summarize some background
on (Hilbert) $C^{*}$-modules and on proper KMS-weights.

In Section 2, we introduce some convenient notation and terminology
related to a $C^{*}$-algebraic analogue of Connes' von
Neumann-algebraic relative tensor product of Hilbert spaces. Our
simple $C^{*}$-relative tensor product is based on the internal tensor
product of $C^{*}$-modules and enjoys all properties that one should
expect like symmetry, associativity, and functoriality.

In Section 3, we introduce a spatial fiber product of $C^{*}$-algebras
that is based on the $C^{*}$-relative tensor product of Hilbert
spaces. The definition is provisional and lacks several desirable
properties like associativity. On the other hand, the construction is functorial,
compatible with the fiber product of von Neumann algebras in a natural
sense, and well-suited for the definition of Hopf
$C^{*}$-bimodules which is given at the end of this section.

In Section 4, we define $C^{*}$-pseudo-multiplicative unitaries, using
the $C^{*}$-relative tensor product of Hilbert spaces introduced in
Section 2. We show that each such unitary is a pseudo-multiplicative
unitary on Hilbert spaces in the sense of Vallin \cite{vallin:2} and
restricts to a pseudo-multiplicative unitary on $C^{*}$-modules in the
sense of Timmermann \cite{timmermann}. To each such unitary, we
associate two algebras and two normal $*$-homomorphisms which, under
favorable circumstances, form Hopf $C^{*}$-bimodules. In particular,
we adapt the regularity condition known for (pseudo-)multiplicative
unitaries \cite{baaj:2,enock:10} to $C^{*}$-pseudo-multiplicative
unitaries and show that if this condition is satisfied, then the
algebras and $*$-homomorphisms mentioned above do form Hopf
$C^{*}$-bimodules.

Section 6 is devoted to locally compact groupoids. We discuss
the $C^{*}$-pseudo-multiplicative unitary associated to a locally
compact groupoid, show that it is regular, and determine the
associated Hopf $C^{*}$-bimodules.

\paragraph{Preliminaries}

Given a subset $Y$ of a normed space $X$, we denote by $[Y] \subset X$
the closed linear span of $Y$.

Given a Hilbert space $H$ and a subset $X \subseteq {\cal L}(H)$, we
denote by $X'$  the commutant of $X$. Given  Hilbert spaces $H$, $K$,
a $C^{*}$-subalgebra $A \subseteq {\cal L}(H)$, and a $*$-homomorphism
$\pi \colon A \to {\cal L}(K)$, we put
\begin{align*}
  {\cal L}^{\pi}(H,K) := \{ T \in {\cal L}(H,K) \mid Ta = \pi(a)T
  \text{ for all } a \in A\};
\end{align*}
thus, for example, $A' = {\cal L}^{\Id_{A}}(H)$.

We shall make extensive use of (right) $C^{*}$-modules, also known as
Hilbert $C^{*}$-modules or Hilbert modules. A standard reference is
\cite{lance}.

All sesquilinear maps like inner products of Hilbert spaces
or $C^{*}$-modules are assumed to be conjugate-linear in the first
component and linear in the second one.

Let $A$ and $B$ be $C^{*}$-algebras.  
Given $C^{*}$-modules $E$ and $F$ over $B$, we denote the space of all
adjointable operators $E\to F$ by ${\cal L}_{B}(E,F)$, and the subspace
of all compact operators by ${\cal K}_{B}(E,F)$.

Let $E$ and $F$ be $C^{*}$-modules over $A$ and $B$, respectively, and
let $\pi \colon A \to {\cal L}_{B}(F)$ be a $*$-homomorphism. Then one
can form the internal tensor product $E \otimes_{\pi} F$, which is a
$C^{*}$-module over $B$ \cite[Chapter 4]{lance}. This $C^{*}$-module
is the closed linear span of elements $\eta \otimes_{A} \xi$, where
$\eta \in E$ and $\xi \in F$ are arbitrary, and $\langle \eta
\otimes_{\pi} \xi|\eta' \otimes_{\pi} \xi'\rangle = \langle
\xi|\pi(\langle\eta|\eta'\rangle)\xi'\rangle$ and $(\eta \otimes_{\pi}
\xi)b=\eta \otimes_{\pi} \xi b$ for all $\eta,\eta' \in E$, $\xi,\xi'
\in F$, and $b \in B$.  We denote the internal tensor product by
``$\tr$''; thus, for example, $E \tr_{\pi} F=E \otimes_{\pi} F$. If
the representation $\pi$ or both $\pi$ and $A$ are understood, we
write ``$\tr_{A}$'' or  ``$\tr$'', respectively, instead of
$"\tr_{\pi}$''. 

Given $E$, $F$ and $\pi$ as above, we define a {\em flipped internal
  tensor product} $F {_{\pi}\tl} E$ as follows. We equip the algebraic
tensor product $F \odot E$ with the structure maps $\langle \xi \odot
\eta | \xi' \odot \eta'\rangle := \langle \xi| \pi(\langle
\eta|\eta'\rangle) \xi'\rangle$, $(\xi \odot \eta) b := \xi b \odot
\eta$, and by factoring out the null-space of the semi-norm
$\zeta\mapsto \| \langle \zeta|\zeta\rangle\|^{1/2}$ and taking
completion, we obtain a $C^{*}$-$B$-module $F {_{\pi}\tl} E$.  This is
the closed linear span of elements $\xi \tl \eta$, where $\eta \in E$
and $\xi \in F$ are arbitrary, and $\langle \xi \tl \eta|\xi'
\tl \eta'\rangle = \langle \xi|\pi(\langle\eta|\eta'\rangle)\xi'\rangle$ and
$(\xi \tl \eta)b=\xi b \tl \eta$ for all $\eta,\eta' \in E$, $\xi,\xi'
\in F$, and $b\in B$. As above, we write ``${_{A}\tl}$'' or simply
``$\tl$'' instead of ``${_{\pi}\tl}$'' if the representation $\pi$ or
both  $\pi$ and $A$ are understood, respectively.

Evidently, the usual and the flipped internal tensor product are related by
a unitary map $\Sigma \colon F \tr E \mycong E \tl F$, $\eta \tr \xi
\mapsto \xi \tl \eta$.

We shall frequently use the following result \cite[Proposition
1.34]{echterhoff}:
\begin{proposition} \label{proposition:itp-operators} Let
  $E_{1},E_{2}$ be $C^{*}$-modules over $A$, let $F_{1}$, $F_{2}$ be
  $C^{*}$-modules over $B$ with representations $\pi_{i} \colon A \to
  {\cal L}_{B}(F_{i})$ ($i=1,2$), and let $S \in {\cal
    L}_{A}(E_{1},E_{2})$, $T \in {\cal L}_{B}(F_{1},F_{2})$ such that
  $T\pi_{1}(a)=\pi_{2}(a)T$ for all $a \in A$. Then there exists a
  unique operator $S \tr T \in {\cal L}_{B}(E_{1} \tr F_{1}, E_{2} \tr
  F_{2})$ such that $(S \tr T)(\eta \tr \xi)= S\eta \tr T\xi$ for all
  $\eta \in E_{1}$ and $\xi \in F_{1}$. Moreover, $(S\tr T)^{*}=S^{*} \tr
  T^{*}$.\qed
\end{proposition}

We shall frequently consider the following kind of $C^{*}$-modules:
Let $H$ and $K$ be Hilbert spaces. We call a subset $\Gamma \subseteq
{\cal L}(H,K)$ a {\em concrete $C^{*}$-module} if $\lnspan \Gamma
\Gamma^{*} \Gamma\rnspan = \Gamma$.  If $\Gamma$ is such a concrete
$C^{*}$-module, then evidently $\Gamma^{*}$ is a concrete
$C^{*}$-module as well, the space $B:=\lnspan \Gamma^{*}
\Gamma\rnspan$ is a $C^{*}$-algebra, and $\Gamma$ is a full right
$C^{*}$-module over $B$ with respect to the inner product given by
$\langle \zeta|\zeta'\rangle=\zeta^{*}\zeta'$ for all $\zeta,\zeta'
\in \Gamma$.
\begin{lemma} \label{lemma:itp-concrete}
  Let $H$, $K$ and $L$ be Hilbert spaces and $\Delta \subseteq {\cal
    L}(H,K)$ and $\Gamma \subseteq {\cal L}(K,L)$ be concrete
  $C^{*}$-modules such that $\lnspan \Gamma^{*}\Gamma \Delta\rnspan
  \subseteq \Delta$. Then $\lnspan \Gamma \Delta\rnspan \subseteq
  {\cal L}(H,L)$ is a concrete $C^{*}$-module, and there exists an
  isomorphism of right $C^{*}$-modules $\Gamma \tr_{\lnspan
    \Gamma^{*}\Gamma\rnspan} \Delta \cong \lnspan
  \Gamma\Delta\rnspan$, $\gamma \tr \delta \mapsto \gamma\delta$.
\end{lemma}
We shall be primarily interested in the case where $H=K$ and 
$\Delta \subseteq {\cal L}(H)$ is a $C^{*}$-algebra.

\section{The $C^{*}$-relative tensor product}

The construction of a ($C^{*}$-)relative tensor product of Hilbert
spaces is needed for the definition of ($C^{*}$-)pseudo-multiplicative
unitaries and almost everywhere in the theory of locally compact
quantum groupoids.

 In the setting of von Neumann algebras, the relative
tensor product is constructed as follows. Given a Hilbert
space $K$ with a (normal, faithful, nondegenerate) antirepresentation
of a von Neumann algebra $N$ equipped with a (normal, faithful,
semifinite) weight $\nu$, one defines a subspace $D(K;\nu) \subseteq
K$ of {\em bounded elements} and an $N$-valued inner product $\langle
\frei|\frei \rangle_{\nu}$ on $D(K;\nu)$. Given another Hilbert space $H$
with a normal faithful nondegenerate representation of $N$, one
defines the relative tensor product $H \rtensor{}{\nu}{} K$ to be the
completion of the algebraic tensor product $H \odot D(K;\nu^{op})$
with respect to the inner product $\langle \eta \odot \xi|\eta \odot
\xi'\rangle = \langle \eta |
\langle\xi|\xi'\rangle_{\nu}\eta'\rangle$. This construction is
symmetric in the sense that one can equivalently define $H
\rtensor{}{\nu}{} K$ as the completion of an algebraic tensor product
$D(H;\nu) \odot K$, where $D(H;\nu) \subseteq H$ is a subspace
equipped with an $N^{op}$-valued inner product.

In this section, we formalize a $C^{*}$-algebraic analogue of this
construction. It seems that in the setting of $C^{*}$-algebras, the
analogues of the spaces $D(K,\nu)$ and $D(H;\nu^{op})$ can not be
reconstructed from an (anti)representation of a $C^{*}$-algebra
equipped with a weight alone, but need to be given explicitly in the
form of $C^{*}$-factorizations (Subsection
\ref{subsection:rtp-factorization}). The definition of the
$C^{*}$-relative tensor product is then straightforward (Subsection
\ref{subsection:rtp-definition}) and essentially coincides with the
relative tensor product in the setting of von Neumann algebras
(Subsection \ref{subsection:rtp-comparison}).

\subsection{$C^{*}$-bases and $C^{*}$-factorizations} 
\label{subsection:rtp-factorization}

To define a $C^{*}$-algebraic analogue of the relative tensor product
$H \rtensor{}{\nu}{} K$ described above, we replace
\begin{itemize}
\item the von Neumann algebra $N$ and the weight $\nu$ by a
  $C^{*}$-base $\epsilon$ (Definition \ref{definition:rtp-base}
  and Example \ref{example:rtp-base-weight}), and
\item the (anti)representations of $N$ on $H$ and $K$ by
  $C^{*}$-(anti)factorizations of $H$ and $K$ (Definition
  \ref{definition:rtp-factorization})  relative to $\epsilon$.
\end{itemize}
\begin{definition} \label{definition:rtp-base} A {\em $C^{*}$-base} is
  a triple $\cbasel{B}{H}$, shortly written $\cbases{B}{H}$,
  consisting of a Hilbert space $\frakH$ and two commuting
  nondegenerate $C^{*}$-algebras $\frakB,\frakB^{\dagger} \subseteq
  {\cal L}(\frakH)$.

  Two $C^{*}$-bases $\cbases{B}{H}$ and $\cbasesc$ are equivalent
  if $\frakC=\Ad_{U}(\frakB)$ and
  $\frakC^{\dagger}=\Ad_{U}(\frakB^{\dagger})$ for some unitary
  $U\colon \frakH \to \frakK$.  
\end{definition}

Every proper KMS-weight on a $C^{*}$-algebra gives rise to a
$C^{*}$-base:
\begin{example} \label{example:rtp-base-weight}
  Let $B$ be a $C^{*}$-algebra with a proper KMS-weight $\mu$.
  As usual, we put $\Nmu:=\{b \in B\mid \mu(b^{*}b) < \infty\}$ and
  denote by $(H_{\mu},\Lambda_{\mu},\pi_{\mu})$ a GNS-construction for
  $\mu$, i.e., $H_{\mu}$ is a Hilbert space, $\Lambda_{\mu} \colon
  \Nmu \to H_{\mu}$ is a linear map with dense image and $\pi_{\mu}
  \colon B \to {\cal L}(H_{\mu})$ is a representation such that
  $\langle \Lambda_{\mu}(b)|\Lambda_{\mu}(b')\rangle=\mu(b^{*}b')$ and
  $\pi_{\mu}(c)\Lambda_{\mu}(b') = \Lambda_{\mu}(cb')$ for all $b,b'
  \in \Nmu$ and $c\in B$. Moreover, we denote by $J_{\mu} \colon
  H_{\mu} \to H_{\mu}$ the modular conjugation, which is a
  conjugate-linear isometric isomorphism.
  Then the $C^{*}$-algebras $\pi_{\mu}(B)$ and
  $J_{\mu}\pi_{\mu}(B)J_{\mu}$ commute and the triple
  $(H_{\mu},\pi_{\mu}(B),J_{\mu}\pi_{\mu}(B)J_{\mu})$
  is a $C^{*}$-base.
  The opposite $C^{*}$-base is
  equivalent to the $C^{*}$-base associated to the opposite
  weight $\mu^{op}$ on $B^{op}$, given by
  $\mu^{op}((b^{op})^{*}b^{op})=\mu(bb^{*})$ for all $b \in
  B$. Indeed, $\mu^{op}$ is a proper KMS-weight,
  $\mathfrak{N}_{\mu^{op}}=(\Nmu^{*})^{op}$, and the triple
  $(H_{\mu^{op}},\Lambda_{\mu^{op}},\pi_{\mu^{op}})$ given by
  \begin{align*}
    H_{\mu^{op}}&:=H_{\mu}, &
    \Lambda_{\mu^{op}}(b^{op})&:=J_{\mu}\Lambda_{\mu}(b^{*}), &
    \pi_{\mu^{op}}(c^{op}) &= J_{\mu}\pi_{\mu}(c^{*}) J_{\mu}
  \end{align*}
  for all $b \in \Nmu^{*}$ and $c \in B$ is a GNS-construction for
  $\mu^{op}$. In particular,
  $\pi_{\mu^{op}}(B^{op})=J_{\mu}\pi_{\mu}(B)J_{\mu}$, and replacing
  $\mu$ by $\mu^{op}$, we find
  $\pi_{\mu}(B)=J_{\mu^{op}}\pi_{\mu^{op}}(B^{op})J_{\mu^{op}}$.
\end{example}

\begin{definition} \label{definition:rtp-factorization}
  A {\em $C^{*}$-factorization} of a Hilbert space $H$ with respect to
  a $C^{*}$-base $\cbases{B}{H}$ is a closed subspace $\alpha
  \subseteq {\cal L}(\frakH,H)$ satisfying $\lnspan
  \alpha^{*}\alpha\rnspan = \frakB$, $\lnspan \alpha \frakB\rnspan =
  \alpha$, and $\lnspan \alpha \frakH \rnspan = H$.  We denote the set
  of all $C^{*}$-factorizations of a Hilbert space $H$ with respect to
  a $C^{*}$-base $\cbases{B}{H}$ by $\cfact(H;\cbases{B}{H})$.
\end{definition}

Let $\alpha$ be a $C^{*}$-factorization of a Hilbert space $H$ with
respect to a $C^{*}$-base $\cbases{B}{H}$. Then $\alpha$
is a concrete $C^{*}$-module and a full right $C^{*}$-module over $\frakB$
with respect to the inner product $\langle \xi|\xi'\rangle
:=\xi^{*}\xi'$.  Moreover, there exists a unitary
\begin{align} \label{eq:rtp-iso}
  \alpha \tr \frakH \mycong H, \quad \xi
      \tr \zeta \mapsto \xi\zeta.
\end{align}
From now on, we shall identify $\alpha \tr \frakH$ with $H$ as above
without further notice. 

By Proposition \ref{proposition:itp-operators}, there exists a
unique representation
\begin{align*}
  \rho_{\alpha} \colon \frakB^{\dagger} \to {\cal L}(\alpha \tr \frakH)
      \cong {\cal L}(H)
\end{align*}
such that for all $b^{\dagger} \in \frakB^{\dagger}$ and $\xi\in
\alpha$, $\zeta \in \frakH$,
\begin{align*}
  \rho_{\alpha}(b^{\dagger})(\xi \tr \zeta) = \xi \tr
  b^{\dagger}\zeta \quad \text{or, equivalently,}  \quad
  \rho_{\alpha}(b^{\dagger})\xi\zeta = \xi b^{\dagger}\zeta.
\end{align*}
Clearly, this representation is nondegenerate and faithful.

 Let $K$ be a  Hilbert space. Then each unitary  $U\colon H
 \to K$ induces a map
 \begin{align*}
   V_{*} \colon \cfact(H;\cbases{B}{H}) \to \cfact(K;\cbases{B}{H}),  \quad \alpha
   \mapsto V\alpha.
 \end{align*}
Let $\beta$ be a $C^{*}$-factorization of $K$ with respect to
$\cbases{B}{H}$. We put
\begin{align*}
  {\cal L}(H_{\alpha},K_{\beta})&:= \big\{ T \in {\cal L}(H,K)
  \,\big|\, T\alpha \subseteq \beta, \, T^{*}\beta \subseteq \alpha\big\}.
\end{align*}
Evidently, ${\cal L}(H_{\alpha},K_{\beta})^{*} = {\cal
  L}(K_{\beta},H_{\alpha})$. Let $T \in {\cal
  L}(H_{\alpha},K_{\beta})$. Then the  map
\begin{align*}
  T_{\alpha} \colon \alpha
  \to \beta, \quad \xi \mapsto T\xi
\end{align*}
is an adjointable operator of $C^{*}$-modules with adjoint
$(T_{\alpha})^{*}=(T^{*})_{\beta}$. With respect to the isomorphism
$\alpha \tr \frakH \cong H$, we have $T \equiv T_{\alpha} \tr
\Id_{\frakH}$ and
\begin{align} \label{eq:rtp-op-commute}
  T \rho_{\alpha}(b^{\dagger}) \equiv (T_{\alpha} \tr \Id_{\frakH})
  (\Id_{\alpha} \tr b^{\dagger}) = (\Id_{\beta} \tr
  b^{\dagger})(T_{\alpha} \tr \Id_{\frakH})\equiv \rho_{\beta}(b^{\dagger})T
\end{align}
for all $b^{\dagger} \in B^{\dagger}$.
 
We shall consider Hilbert spaces equipped with several
$C^{*}$-factorizations that are compatible in the following sense:
\begin{definition} \label{definition:rtp-compatible} Let $\alpha$ be a
  $C^{*}$-factorization of a Hilbert space $H$ with respect to a
  $C^{*}$-base $\cbases{B}{H}$, and let $\cbasesc$ be another
  $C^{*}$-base. We call a $C^{*}$-factorization $\beta \in
  \cfact(H;\cbasesc)$ {\em compatible} with $\alpha$, written
  $\alpha \perp \beta$, if $\lnspan
  \rho_{\alpha}(\frakB^{\dagger})\beta\rnspan =\beta $ and $\lnspan
  \rho_{\beta}(\frakC^{\dagger})\alpha\rnspan=\alpha$, and put
  \begin{align*}
   \cfact(H_{\alpha};\cbasesc):=\{ \beta \in
  \cfact(H;\cbasesc) \mid \alpha \perp \beta\}.
  \end{align*}
\end{definition}
\begin{example}
  If $\cbases{B}{H}$ is a $C^{*}$-base, then
  \begin{align*}
    \frakB &\in
  \cfact(\frakH;\cbases{B}{H}), 
\
  \rho_{\frakB} = \Id_{\frakB^{\dagger}}, &
  \frakB^{\dagger} &\in \cfact(\frakH;\cbaseos{B}{H}), \
  \rho_{\frakB^{\dagger}} =\Id_{\frakB},
  \end{align*}
and hence $\frakB \perp \frakB^{\dagger}$.
\end{example}
\begin{remark} \label{remark:rtp-commute}
  Let $H$, $\cbases{B}{H}$, $\cbasesc$ and $\alpha$, $\beta$ be as
  in Definition \ref{definition:rtp-compatible}. If $\alpha \perp
  \beta$, then $\rho_{\alpha}(\frakB^{\dagger}) \subseteq
  {\cal L}(H_{\beta})$, $\rho_{\beta}(\frakC^{\dagger})
  \subseteq {\cal L}(H_{\alpha})$, and Equation
  \eqref{eq:rtp-op-commute} implies that
  $\rho_{\alpha}(\frakB^{\dagger})$ and
  $\rho_{\beta}(\frakC^{\dagger})$ commute.
\end{remark}

\subsection{Definition and basic properties}
\label{subsection:rtp-definition}

Assume that we are given the following data:
\begin{gather} \label{eq:rtp-data}
\begin{gathered}
  \text{Hilbert spaces } H,K, \ \text{a $C^{*}$-base } \cbasesb, \
  \text{and} \\
  \text{$C^{*}$-factorizations } \alpha \in \cfact(H;\cbasesb), \
  \beta \in \cfact(H;\cbaseosb)
  \end{gathered}
\end{gather}
 Then we can form the internal tensor product
\begin{align*}
\HfibreK := \alpha \tr \frakH \tl \beta,
\end{align*}
and the isomorphism \eqref{eq:rtp-iso} induces isomorphisms
\begin{gather} \label{eq:rtp-space}
    \alpha \tr_{\rho_{\beta}} K \cong \HfibreK \cong
    H_{\rho_{\alpha}} \tl \beta, \quad
    \xi \tr \eta \zeta \equiv \xi \tr \zeta \tl \eta \equiv \xi \zeta
    \tl \eta.
\end{gather}
From now on, we shall use these isomorphisms without further
mentioning.

The definition of $\HfibreK$ is functorial in the following sense. Let
$L$, $M$ be Hilbert spaces,  $\gamma \in \cfact(L;\cbasesb)$,
$\delta \in \cfact(M;\cbaseosb)$, and $S \in {\cal
  L}(H,L)$, $T \in {\cal
  L}(K,M)$.  We define an operator
\begin{align*}
 S \rtensorh T\in {\cal L}\big(\HfibreK, \, L
    \rtensor{\gamma}{\frakH}{\delta} M\big) 
\end{align*}
in the following cases, using Proposition
\ref{proposition:itp-operators} and the isomorphisms \eqref{eq:rtp-space}:
\begin{enumerate}
\item If $S \in {\cal L}(H_{\alpha},L_{\gamma})$ and
  $T\rho_{\beta}(b)=\rho_{\delta}(b)T$ for all $b \in \frakB$, we put
  $S \rtensorh T \equiv S_{\alpha} \tr T \in {\cal L}(\alpha
  \tr_{\rho_{\beta}} K, \gamma \tr_{\rho_{\delta}} M)$.
\item  If $T \in {\cal L}(K_{\beta},M_{\delta})$ and
  $S\rho_{\alpha}(b^{\dagger}) = \rho_{\gamma}(b^{\dagger})S$ for all
  $b^{\dagger} \in \frakB^{\dagger}$, we put $S \rtensorh T \equiv S
  \tl T_{\beta} \in {\cal L}(H {_{\rho_{\alpha}}\tl} \beta, L
  {_{\rho_{\gamma}}\tl} M)$.
\end{enumerate}
If $S \in {\cal L}(H_{\alpha},L_{\gamma})$ and $T \in {\cal
  L}(K_{\beta},M_{\delta})$, we get $S \rtensorh T = S_{\alpha} \tr
\Id_{\frakH} \tl T_{\beta}$ in both cases. 
In the special case where $S=\Id_{H}$ or $T=\Id_{K}$, we abbreviate
\begin{align*}
  S_{\leg{1}} &:=  S \rtensorh \Id \colon\HfibreK \to L
  \rtensor{\gamma}{\frakH}{\beta} K, \\
  T_{\leg{2}} &:= \Id \rtensorh T \colon  \HfibreK \to H
  \rtensor{\alpha}{\frakH}{\delta} M.
\end{align*}

Given $C^{*}$-algebras $C$, $D$ and $*$-homomorphisms $\rho \colon C
\to \rho_{\alpha}(\frakB^{\dagger})' \subseteq {\cal L}(H)$, $\sigma
\colon D \to \rho_{\beta}(\frakB)' \subseteq {\cal L}(K)$, we put
\begin{align*}
  \rho_{\leg{1}} \colon C \to {\cal L}(\HfibreK), \ c \mapsto
  \rho(c)_{\leg{1}} = \rho(c) \tl \Id_{\beta}, \\
\sigma_{\leg{2}} \colon D\to {\cal L}(\HfibreK), \ d \mapsto
\sigma(d)_{\leg{2}} = \Id_{\alpha} \tr \sigma(d).
\end{align*}
Combining the leg notation and the ket-bra notation, we define 
for each $\xi \in \alpha$ and $\eta \in \beta$ two pairs
of adjoint operators
\begin{align*}
  |\xi\rangle_{\leg{1}} \colon K &\to \HfibreK, \ \zeta \mapsto \xi
  \tr \zeta, & \langle \xi|_{\leg{1}}:=|\xi\rangle_{\leg{1}}^{*}\colon
  \xi' \tr \zeta &\mapsto
  \rho_{\beta}(\langle\xi|\xi'\rangle)\zeta, \\
  |\eta\rangle_{\leg{2}} \colon H &\to \HfibreK, \ \zeta \mapsto \zeta
  \tl \eta, & \langle\eta|_{\leg{2}} := |\eta\rangle_{\leg{2}}^{*}
  \colon \zeta \tl\eta &\mapsto \rho_{\alpha}(\langle
  \eta|\eta'\rangle)\zeta.
\end{align*}
We put $\kalpha{1} := \big\{ |\xi\rangle_{\leg{1}} \,\big|\, \xi \in
\alpha\big\}$ and similarly define $\balpha{1}$, $\kbeta{2}$,
$\bbeta{2}$.

\begin{proposition} \label{proposition:rtp-factorization-maps}
  Let $H, K, \cbases{B}{H}, \alpha, \beta$ be as in
  \eqref{eq:rtp-data}, and let $\cbasesc$, $\cbases{D}{L}$ be
  $C^{*}$-bases. Then there exist compatibility-preserving maps
    \begin{align*}
      \cfact(H_{\alpha};\cbasesc) &\to
      \cfact(\HfibreK;\cbasesc), & \gamma &\mapsto \gamma \lt
      \beta:= \lnspan
      \kbeta{2}\gamma\rnspan, \\
      \cfact(K_{\beta};\cbases{D}{L}) &\to
      \cfact(\HfibreK;\cbases{D}{L}), & \delta &\mapsto \alpha \rt
      \delta:= \lnspan \kalpha{1}\delta\rnspan,
    \end{align*}
    For all $\gamma \in
    \cfact(H_{\alpha};\cbasesc)$ and $\delta \in
    \cfact(K_{\beta};\cbases{D}{L})$,
    \begin{align*}
   \rho_{(\gamma \lt \beta)}
    &=(\rho_{\gamma})_{\leg{1}}, & \rho_{(\alpha \rt \delta)} &=
    (\rho_{\delta})_{\leg{2}}, & \gamma \lt \beta &\perp \alpha
    \rt \delta.
    \end{align*}
\end{proposition}
\begin{proof}
  Let $\gamma \in \cfact(H_{\alpha};\cbasesc)$. Then $\gamma \lt
  \beta = \lnspan \kbeta{2}\gamma\rnspan \subseteq {\cal
    L}(\frakK,\HfibreK)$ is a $C^{*}$-factorization of $\HfibreK$
  with respect to $\cbasesc$ because
  \begin{gather*}
    \lnspan \gamma^{*} \bbeta{2} \kbeta{2} \gamma\rnspan = \lnspan
    \gamma^{*} \rho_{\beta}(\frakB)\gamma\rnspan = \lnspan
    \gamma^{*}\gamma\rnspan = \frakC, \\  \lnspan
    \kbeta{2}\gamma\frakC\rnspan =\lnspan
    \kbeta{2}\gamma\rnspan, \quad
    \lnspan \kbeta{2}\gamma \frakK\rnspan = \lnspan \kbeta{2}
    H\rnspan = \HfibreK.
  \end{gather*}
  The relation $\rho_{(\gamma \lt \beta)}=(\rho_{\gamma})_{\leg{1}}$
  is evident.  If $\gamma' \in \cfact(H_{\alpha};\cbasesc)$ and
  $\gamma' \perp \gamma$, then $\gamma' \lt \beta \perp \gamma \lt
  \beta$ because
  \begin{align*}
    \lnspan \rho_{(\gamma \lt \beta)}(\frakC^{\dagger}) \kbeta{2}\gamma'\rnspan
    = \lnspan (\rho_{\gamma})_{\leg{1}}(\frakC^{\dagger})
    \kbeta{2}\gamma'\rnspan = \lnspan \kbeta{2}
    \rho_{\gamma}(\frakC^{\dagger})\gamma'\rnspan = \lnspan \kbeta{2}
    \gamma'\rnspan
  \end{align*}
  and similarly $\lnspan \rho_{(\gamma' \lt \beta)}(\frakC^{\dagger})
  \kbeta{2}\gamma\rnspan = \lnspan \kbeta{2} \gamma\rnspan$.

  Let $\delta \in \cfact(K_{\beta};\cbases{D}{L})$.  Similar
  arguments as above show that $\alpha \rt \delta \in
  \cfact(\HfibreK;\cbases{D}{L})$, that $\rho_{(\alpha\rt
    \delta)}=(\rho_{\delta})_{\leg{1}}$, and that the map $\delta'
  \mapsto \alpha \rt \delta'$ preserves compatibility.
  Finally, $\gamma \lt \beta \perp \alpha \lt \delta$ because
  \begin{align*}
    \lnspan \rho_{(\gamma \lt \beta)}(\frakC^{\dagger})
    \kalpha{1}\delta\rnspan = \lnspan
    (\rho_{\gamma})_{\leg{1}}(\frakC^{\dagger})
    \kalpha{1}\delta\rnspan = \lnspan
    |\rho_{\gamma}(\frakC^{\dagger})\alpha\rangle_{\leg{1}}\delta\rnspan
    = \lnspan \kalpha{1}\delta\rnspan
  \end{align*}
  and similarly $\lnspan \rho_{(\alpha \rt
    \delta)}(\frakD^{\dagger})\kbeta{2}\gamma\rnspan = \lnspan
  \kbeta{2}\gamma\rnspan$.
\end{proof}

Evidently, the relative tensor product is symmetric in the following sense:
\begin{proposition} \label{proposition:rtp-symmetric}
  Let $H, K, \cbases{B}{H}, \alpha, \beta$ be as in
  \eqref{eq:rtp-data}.  There exists a unitary
    \begin{align*}
      \Sigma\colon \HfibreK \cong \alpha \tr \frakH \tl \beta &\to
      \beta
      \tr \frakH \tl \alpha \cong K \rtensor{\beta}{\frakH}{\alpha} H, \\
      \xi \tr \zeta \tl \eta &\mapsto \eta \tr \zeta \tl \xi.
    \end{align*}
    For each $C^{*}$-base $\cbasesc$ and all $\gamma \in
    \cfact(H_{\alpha};\cbasesc)$, $\delta \in
    \cfact(K_{\beta};\cbasesc)$,\usetagform{empty}
    \begin{align*}
      \Sigma_{*}(\gamma \lt \beta) &= \beta \rt \gamma, &
      \Sigma_{*}(\alpha \rt \delta) &=\delta \lt \alpha. \tag{\qed}
    \end{align*}\usetagform{default}
\end{proposition}

The space $\frakH$ together with the $C^{*}$-factorizations $\frakB$ or
$\frakB^{\dagger}$, respectively, is a left or right unit for the
relative tensor product:
\begin{proposition} \label{proposition:rtp-unital}
  Let $H, K, \cbases{B}{H}, \alpha, \beta$ be as in
  \eqref{eq:rtp-data}.   There exist unitaries
\begin{align*}
  \Phi \colon \frakH \rtensor{\frakB}{\frakH}{\beta} K &\to K, &
  \Psi \colon H \rtensor{\alpha}{\frakH}{\frakB^{\dagger}} \frakH
  &\to H, \\ b \tr \zeta \tl \eta &\mapsto \eta b\zeta, & \xi \tr
  \zeta \tl b^{\dagger} &\mapsto \xi b^{\dagger} \zeta.
\end{align*}
For each $C^{*}$-base $\cbasesc$ and all $\gamma \in
\cfact(H_{\alpha}; \cbasesc)$, $\delta \in
\cfact(K_{\beta};\cbasesc)$,
\begin{align*} 
  \Phi_{*}(\frakB^{\dagger} \lt \beta) &= \beta, & \Phi_{*}(\frakB \rt \delta)
  &= \delta, & \Psi_{*}(\alpha \rt \frakB) &= \alpha, &
  \Psi_{*}(\gamma \lt \frakB^{\dagger}) &= \gamma. 
\end{align*}
\end{proposition}
\begin{proof}
  $\Phi$ and $\Psi$ are obtained by combining \eqref{eq:rtp-iso} and
  \eqref{eq:rtp-space}. The formulas for $\Phi_{*}$ hold because
  $\Phi_{*}(\frakB^{\dagger} \lt \beta) = \lnspan \Phi
  \kbeta{2}\frakB^{\dagger}\rnspan = \lnspan \beta
  \frakB^{\dagger}\rnspan = \beta$ and $ \Phi_{*}(\frakB \rt \delta) =
  \lnspan \Phi |\frakB\rangle_{\leg{1}}\delta\rnspan = \lnspan
  \rho_{\beta}(\frakB)\delta\rnspan = \delta$; the formulas for
  $\Psi_{*}$ follow similarly.
\end{proof}

The relative tensor product is functorial in the following sense:
\begin{proposition}\label{proposition:rtp-functorial}
  Let $H, K, \cbases{B}{H}, \alpha, \beta$ be as in
  \eqref{eq:rtp-data}, $\cbasesc$ a $C^{*}$-base, and $\gamma \in
  \cfact(H_{\alpha};\cbasesc)$, $\delta \in
  \cfact(K_{\beta};\cbasesc)$.  Then
  \begin{align*}
    S \rtensorh T &\in {\cal L}\big((\HfibreK)_{\gamma \lt \beta}\big)
    \ \text{for all } S \in \rho_{\alpha}(\frakB^{\dagger})' \cap
    {\cal L}(H_{\gamma}), \, T \in
    {\cal L}(K_{\beta}), \\
    S \rtensorh T &\in {\cal L}\big((\HfibreK)_{\alpha \rt \delta}\big)
    \ \text{for all } S\in {\cal L}(H_{\alpha}), \, T \in
    \rho_{\beta}(\frakB)' \cap {\cal L}(K_{\delta}).
  \end{align*}
\end{proposition}
\begin{proof}
  We only prove the first inclusion; the second one follows similarly.
  For all $S \in \rho_{\alpha}(\frakB^{\dagger})' \cap {\cal
    L}(H_{\gamma})$ and $T \in {\cal L}(K_{\beta})$, we have $(S
  \rtensorh T)(\gamma \lt \beta) = \lnspan
  |T_{\beta}\beta\rnspan_{\leg{2}} S_{\gamma}\gamma\rnspan \subseteq
  \lnspan \kbeta{2}\gamma\rnspan=\gamma \lt \beta$ and similarly $(S
  \rtensorh T)^{*}(\gamma \lt \beta) \subseteq \gamma \lt \beta$.
\end{proof}

The relative tensor product is associative in the following sense:
\begin{proposition} \label{proposition:rtp-associative}
  Let $H,K,L$ be Hilbert spaces, $\cbasesb$, $\cbasesc$, $
  \cbases{D}{L}$ $C^{*}$-bases, and $\alpha \in \cfact(H;\cbasesb)$, $
  \beta \in \cfact(K;\cbaseosb)$, $ \gamma \in
  \cfact(K;\cbasesc)$, $ \delta \in \cfact(L;\cbaseos{C}{K})$
  such that $\beta \perp \gamma$. Then there exists an isomorphism
    \begin{align*}
      \Theta \colon (\HfibreK) \rtensor{\alpha \rt
        \gamma}{\frakK}{\delta} L &\to (\alpha\tr_{\rho_{\beta}} K)
      {_{\rho_{(\alpha \rt \gamma)}}\tl}
      \delta  \\
      &\to \alpha \tr_{\rho_{\beta}} K {_{\rho_{\gamma}}}\tl \delta \\
      &\to \alpha \tr_{\rho_{(\beta \lt \delta)}} (K
      {_{\rho_{\gamma}}\tl} \delta) \to H \rtensor{\alpha}{\frakH}{\beta
        \lt \delta} (K \rtensor{\gamma}{\frakK}{\delta} L),
    \end{align*}
and $\Theta_{*}$ is given by\usetagform{empty}
\begin{align*}
(\epsilon \lt \beta) \lt \delta &\mapsto \epsilon \lt (\beta
  \lt \delta) 
 \text{ for all } \epsilon \in
  \cfact(H_{\alpha};\cbases{D}{L}), \\ (\alpha \lt
  \epsilon) \lt \delta &\mapsto \alpha \rt (\epsilon \lt \delta) 
  \text{f or all } \epsilon \in \cfact(K;\cbases{D}{L}) \text{ s.t.\ }
  \beta\perp \epsilon\perp \gamma, \\ (\alpha
  \rt \gamma) \rt \epsilon &\mapsto \alpha \rt (\gamma \rt \epsilon) 
  \text{ for all } \epsilon \in \cfact(L_{\delta};\cbases{D}{L}).  \tag{\qed}
\end{align*}\usetagform{default}
\end{proposition}

\subsection{Relation to the relative tensor product}
\label{subsection:rtp-comparison}

Let $H,K,\cbasesb,\alpha,\beta$ be as in \eqref{eq:rtp-data}.  If
$\cbasesb$ arises from a proper KMS-weight $\mu$ on a $C^{*}$-algebra
$B$ as in Example \ref{example:rtp-base-weight}, then $\HfibreK$ can be
identified with a von Neumann-algebraic relative tensor product as follows.

We use the same notation as in Example \ref{example:rtp-base-weight}.  The
proper KMS-weight $\mu$ extends to a normal semifinite faithful weight
$\tilde\mu$ on the von Neumann algebra $N:=\frakB''$, and the
maps $\Lambda_{\mu}$ and $\pi_{\mu}$ extend uniquely to maps
$\Lambda_{\tilde \mu}$ and $\pi_{\tilde \mu}$ such that
$(H_{\mu},\Lambda_{\tilde \mu},\pi_{\tilde \mu})$ becomes a
GNS-construction for $\tilde\mu$. As in the case of
$C^{*}$-algebras, we obtain from this GNS-construction and the modular
conjugation $J_{\mu}=J_{\tilde \mu}$ a GNS-construction
$(H_{\mu},\Lambda_{\tilde \mu^{op}},\pi_{\tilde \mu^{op}})$ for the
opposite weight $\tilde \mu^{op}$ on $N^{op}$.

Since $\pi_{\tilde\mu^{op}}(N^{op})$ commutes with
$\frakB=\pi_{\mu}(B)$ and $\pi_{\tilde\mu}(N)$ commutes with
$\frakB^{\dagger}=\pi_{\mu^{op}}(B^{op})$, respectively, we can extend
$\rho_{\alpha}$ and $\rho_{\beta}$ to representations
\begin{align*}
  \tilde \rho_{\alpha} \colon N^{op} &\to {\cal L}(\alpha \tr \frakH)
  \cong {\cal L}(H), & \tilde \rho_{\beta} \colon N &\to {\cal
    L}(\beta \tr \frakH)
  \cong {\cal L}(K), \\
  y &\mapsto \Id_{\alpha} \tr \pi_{\tilde\mu^{op}}(y^{op}), & x
  &\mapsto \Id_{\beta} \tr \pi_{\tilde\mu}(x).
\end{align*}
\begin{lemma}
  The representations $\tilde \rho_{\alpha}$ and $\tilde \rho_{\beta}$
  are faithful, nondegenerate, and normal.
\end{lemma}
\begin{proof}
  We only prove the assertions concerning $\tilde \rho_{\alpha}$. This
  representation is faithful and nondegenerate because $\frakB$ is
  nondegenerate. Let us show that $\tilde \rho_{\alpha}$ is
  normal. Every normal linear functional on ${\cal L}(H)$ can be
  approximated in norm by functionals of the form $\omega = \langle
  \xi\eta|\frei \xi'\eta'\rangle$, where $\xi,\xi' \in \alpha$ and
  $\zeta,\zeta' \in \frakH$. Therefore, it suffices to show that for
  each such $\omega$, the composition $\omega \circ \tilde
  \rho_{\alpha}$ is normal. But this holds because $(\omega
  \circ \tilde \rho_{\alpha})(x) = \langle \xi \eta| \tilde
  \rho_{\alpha}(x)\xi' \eta'\rangle = \langle \xi \eta|\xi' x
  \eta'\rangle = \langle \xi'^{*}\xi \eta|x\eta'\rangle$
  for all $x \in N$.
\end{proof}
The definition of the von-Neumann-algebraic relative tensor product
involves the subspace
\begin{align*}
  D(H_{\tilde \rho_{\alpha}};\tilde\mu^{op}) := \big\{ \zeta \in H
  \,\big|\, \exists C > 0 \forall y \in \mathfrak{N}_{\tilde \mu^{op}}
  : \|\tilde\rho_{\alpha}(y) \zeta\| \leq C \|\Lambda_{\tilde
    \mu^{op}}(y)\| \big\}.
\end{align*}
Evidently, an element $\zeta \in H$ belongs to $D(H_{\tilde
  \rho_{\alpha}};\tilde\mu^{op})$ if and only if the map
$\Lambda_{\tilde\mu^{op}}(\mathfrak{N}_{\tilde\mu^{op}}) \to H$ given
by $\Lambda_{\tilde\mu^{op}}(y) \mapsto \rho_{\tilde\alpha}(y)\zeta$
extends to a bounded linear map $L(\zeta)\colon H_{\mu} \to H$.
\begin{lemma}
  $\alpha \Lambda_{\tilde \mu}(\mathfrak{N}_{\tilde
    \mu}) \subseteq D(H_{\tilde \rho_{\alpha}};\tilde\mu^{op})$ and
  $L\big(\xi \Lambda_{\tilde \mu}(x)\big) = \xi
  \pi_{\mu}(x)$ for all $\xi \in \alpha$ and $x \in \mathfrak{N}_{\tilde\mu}$.
\end{lemma}
\begin{proof}
  By Tomita-Takesaki theory, $\pi_{\mu}(x) \Lambda_{\tilde\mu^{op}}(y) =
  \pi_{\tilde\mu^{op}}(y) \Lambda_{\tilde\mu}(x)$ for all $x \in
  \mathfrak{N}_{\tilde \mu}$ and $y \in \mathfrak{N}_{\tilde
    \mu^{op}}$. Consequently, 
  \begin{align*}
    \xi \pi_{\mu}(x) \Lambda_{\tilde\mu^{op}}(y) = \xi
    \pi_{\tilde\mu^{op}}(y)\Lambda_{\tilde\mu}(x) =
    \rho_{\tilde\alpha}(y) \xi \Lambda_{\tilde\mu}(x) \quad \text{for
      all } \xi \in \alpha,\, x\in \mathfrak{N}_{\tilde\mu}.
  \end{align*}
  The claims follow.
\end{proof}

Recall that the relative tensor product $\vHfibreK$ of $H$ and $K$
with respect to the representations $\tilde \rho_{\alpha}$,
$\tilde\rho_{\beta}$ and the weight $\tilde\mu$ is the Hilbert space
obtained from the algebraic tensor product $D(H_{\tilde
  \rho_{\alpha}};\tilde\mu^{op}) \odot K$ and the sesquilinear form
given by 
\begin{align*}
  \langle \zeta \odot \omega|\zeta' \odot \omega'\rangle:= \big\langle
  \omega\big|
  \tilde\rho_{\beta}\big(L(\zeta)^{*}L(\zeta)\big)\omega'\big\rangle
\end{align*}
for all $\zeta,\zeta' \in D(H_{\tilde\rho_{\alpha}};\tilde\mu^{op})$
and $\omega,\omega' \in K$.  For all $\zeta \in
D(H_{\tilde\rho_{\alpha}};\tilde\mu^{op})$ and $\omega \in K$, we
denote the image of $\zeta \odot \omega$ in $\vHfibreK$ by $\zeta
\bar{\otimes} \omega$.
\begin{proposition}\label{proposition:rtp-comparison}
  There exists a unique unitary
  \begin{align*}
    \Phi\colon \vHfibreK &\to H
    {_{\rho_{\alpha}}\tl} \beta\cong \HfibreK  \cong  \alpha
    \tr_{\rho_{\beta}} K
  \end{align*}
  such that
  \begin{align} \label{eq:rtp-iso-1} 
    \Phi(\theta \bar{\otimes} \eta\zeta)&\equiv
    L(\theta)\zeta \tl \eta && \text{for all }\theta \in
    D(H_{\tilde\rho_{\alpha}};\tilde\mu^{op}), \, \eta\in \beta, \,
    \zeta\in \frakH, \\
    \label{eq:rtp-iso-2}
    \Phi\big(\xi\Lambda_{\mu}(x) \bar{\otimes} \omega\big)&\equiv
    \xi\pi_{\mu}(x) \tr \omega && \text{for all } \xi \in \alpha, \,
    x\in \mathfrak{N}_{\mu}, \, \omega \in K.
  \end{align}
\end{proposition}
\begin{proof}
  Uniqueness follows from the relation $\lnspan \beta \frakH \rnspan =
  K$ and \eqref{eq:rtp-iso-1}.  Existence of a unitary $\Phi$
  satisfying \eqref{eq:rtp-iso-1} follows from the relation
  \begin{align*}
    \big\langle \theta \bar{\otimes}\eta\zeta\big|\theta'
    \bar{\otimes} \eta'\zeta' \big\rangle_{(\vHfibreK)} &=
    \big\langle \eta\zeta \big|
    \tilde\rho_{\beta}\big(L(\theta)^{*}L(\theta)\big)\eta'\zeta'
    \big\rangle_{K} \\ &= \big\langle L(\theta)\zeta \tl \eta\big|
    L(\theta')\zeta' \tl \eta'\big\rangle_{(H {_{\rho_{\alpha}}\tl}
      \beta)},
  \end{align*}  
  valid for all $\theta,\theta'\in
  D(H_{\tilde\rho_{\alpha}};\mu^{op})$, $\eta,\eta'\in \beta$,
  $\zeta,\zeta' \in H_{\mu}$, and from the relation $\lnspan
  L(D(H_{\tilde\rho_{\alpha}};\tilde\mu^{op}))\frakH\rnspan = H$.
  Formula \eqref{eq:rtp-iso-2} follows from the calculation
  \begin{gather*}
    \Phi\big(\xi\Lambda_{\mu}(x) \bar{\otimes} \eta\zeta\big) \equiv
    L(\xi\Lambda_{\mu}(x))\zeta \tl \eta = \xi\pi_{\mu}(x)\zeta \tl
    \eta \equiv \xi\pi_{\mu}(x) \tr \eta\zeta. \qedhere
  \end{gather*}
\end{proof}

\section{The spatial fiber product of $C^{*}$-algebras}

In this section, we introduce a spatial fiber product of
$C^{*}$-algebras using the $C^{*}$-relative tensor product discussed
above. More precisely, assume that $H,K,\cbasesb,\alpha,\beta$ are as
in \eqref{eq:rtp-data}, so that we can form the $C^{*}$-relative
tensor product $\HfibreK$. Moreover, assume that $A \subseteq {\cal
  L}(H)$ and $B \subseteq {\cal L}(K)$ are $C^{*}$-subalgebras
satisfying $\lnspan A \rho_{\alpha}(\frakB^{\dagger})\rnspan \subseteq
A$ and $\lnspan B \rho_{\beta}(\frakB)\rnspan \subseteq B$. Then the
spatial fiber product of $A$ and $B$ relative to
$\alpha,\beta,\cbasesb$ will be a $C^{*}$-algebra
\begin{align*}
  \AfibreB :=\Ind_{\kalpha{1}}(B) \cap \Ind_{\kbeta{2}}(A) \subseteq
  {\cal L}(\HfibreK),
\end{align*}
where $\Ind_{\kalpha{1}}(B)$ and $\Ind_{\kbeta{2}}(A)$ are obtained by
``inducing up'' $B$ and $A$, respectively, in a sense explained in
Subsection \ref{subsection:fp-induction}.

Like the $C^{*}$-relative tensor product, the fiber product of
$C^{*}$-algebras is an analogue of a classical construction in the
setting of von Neumann algebras.  

Our spatial fiber product of $C^{*}$-algebras lacks several desirable
properties, e.g., associativity, and raises several natural
questions that we can not answer yet.  Nevertheless, it
serves our main purpose --- to describe the target of the
comultiplication of a (concrete) Hopf $C^{*}$-bimodule, in particular
for the legs of a regular $C^{*}$-pseudo-multiplicative unitary. Most
importantly, the construction is functorial with respect to a natural
class of morphisms and ``sufficiently associative'' so that we can
formulate a coassociativity condition for the comultiplication of a
(concrete) Hopf $C^{*}$-bimodule.

\subsection{Induction of $C^{*}$-algebras via $C^{*}$-modules}
\label{subsection:fp-induction}

Let $H$ and $K$ be Hilbert spaces, $\Gamma \subseteq {\cal L}(H,K)$ a
concrete $C^{*}$-module satisfying $\lnspan \Gamma H\rnspan = K$, and
put $\frakB:=\lnspan \Gamma^{*}\Gamma \rnspan \subseteq {\cal L}(H)$
and $\frakC:=\lnspan \Gamma \Gamma^{*}\rnspan \subseteq {\cal L}(K)$.

Let $A \subseteq {\cal L}(H)$ be a nondegenerate $C^{*}$-algebra
satisfying $[A\frakB]\subseteq A$. Put
\begin{align*}
  \indgA &:=\big\{ T \in {\cal L}(K) \,\big|\, T\Gamma, T^{*}\Gamma
  \subseteq [\Gamma A]\big\} \subseteq {\cal L}(K).
\end{align*}
Evidently, $\indgA$ is a $C^{*}$-subalgebra of ${\cal L}(K)$.
Equivalently, this $C^{*}$-algebra can be described as follows.  Since
$\lnspan \frakB A\rnspan \subseteq A$, we can form the
internal tensor product $\Gamma \tr A$.  For each $\gamma \in \Gamma$,
 define $|\gamma\rangle_{\leg{1}} \in {\cal L}_{A}(A, \Gamma \tr A)$
by $a \mapsto \gamma \tr a$, and put $|\Gamma\rangle_{\leg{1}}:=\{
|\gamma\rangle_{\leg{1}} | \gamma \in \Gamma\} \subseteq {\cal
  L}_{A}(A,\Gamma \tr A)$. Finally, let
\begin{align*}
  \indgA' &:= \big\{ T \in {\cal L}_{A}(\Gamma \tr A) \,\big|\,
  T|\Gamma\rangle_{\leg{1}},
  T^{*}|\Gamma\rangle_{\leg{1}} \subseteq [|\Gamma\rangle_{\leg{1}}
  A]\big\}.
\end{align*}
Since $\lnspan \Gamma A H\rnspan = K$, we have an isomorphism $\Gamma
\tr A \tr H \to K$ given by $\gamma \tr a \tr \zeta \mapsto \gamma
a\zeta$.  Using this isomorphism, we define an embedding
\begin{align*}
\iota\colon  {\cal L}_{A}(\Gamma \tr A) \to   {\cal L}(\Gamma
  \tr A \tr H) \cong {\cal L}(K), \ T \mapsto T \tr \Id_{H} \equiv \iota(T).
\end{align*}
\begin{lemma} \label{lemma:ind-iso}
  $\iota$ restricts to a  $*$-isomorphism $\indgA' \to \indgA$.
\end{lemma}
\begin{proof}
  Clearly, $\iota$ defines an embedding $\indgA' \to \indgA$. We show
  that this embedding is surjective. Let $T \in \indgA$. Since $\Gamma
  \tr A$ embeds into ${\cal L}(H,K)$ via $\gamma \tr a \mapsto\gamma
  a$ and since $T \Gamma A \subseteq \lnspan \Gamma A\rnspan$, left
  multiplication by $T$ on $\lnspan \Gamma A\rnspan$ corresponds to a
  map $\tilde T \colon \Gamma \tr A \to \Gamma \tr A$.  One easily
  verifies  $\tilde T \in \indgA'$ and $\iota(\tilde T)=T$.
\end{proof}

Let us give two more equivalent descriptions of the $C^{*}$-algebra
$\indgA$.  Denote by $\tau_{\Gamma}$ the locally convex topology on
${\cal L}(K)$ induced by the family of semi-norms $\big(T \mapsto \|
T\gamma\|\big)_{\gamma \in \Gamma}$ and $\big(T \mapsto \|
T^{*}\gamma\|\big)_{\gamma \in \Gamma}$. Given a subset $X \subseteq
{\cal L}(K)$, denote by $\lnspan X\rnspan_{\Gamma}$ the
$\tau_{\Gamma}$-closed linear span of $X$.
\begin{lemma}
  For every   bounded approximate unit  $(u_{i})_{i}$ of
  the $C^{*}$-algebra $\frakC$ and every $T \in \indgA$, the net
  $\big(u_{i}Tu_{i}\big)_{i}$ converges to $T$ w.r.t.\
  $\tau_{\Gamma}$. 
\end{lemma}
\begin{proof}
  Let $(u_{i})_{i}$ and $T$ as above and $\gamma \in \Gamma$. Then
  \begin{align*} 
    \| u_{i} T u_{i} \gamma - T\gamma\| \leq (\sup_{i} \|u_{i}\|) \|T\|
    \|u_{i}\gamma - \gamma\| - \|u_{i} T\gamma - T\gamma\|
  \end{align*}
  for all $i$, and using $T\gamma \subseteq \lnspan \Gamma A \rnspan$
  and $\lim_{i} u_{i}\gamma'= \gamma'$ for all $\gamma' \in \Gamma$,
  we conclude $\lim_{i} \| u_{i}Tu_{i}\gamma - T\gamma\|=0$.
  Similarly, $\lim_{i} \| u_{i}^{*} T^{*}u_{i}^{*} \gamma -
  T^{*}\gamma\| = 0$.
\end{proof}
\begin{lemma}
$\indgA = \lnspan \Gamma A \Gamma^{*}\rnspan_{\Gamma}$.
\end{lemma}
\begin{proof}
  ``$\subseteq$'': Let $T \in \indgA$. Choose an approximate unit
  $(u_{i})_{i}$ for $\frakC$ and put $T_{i}:=u_{i} T u_{i}$ for each
  $i$. Then $(T_{i})_{i}$ converges to $T$ w.r.t.\ $\tau_{\Gamma}$,
  and $T_{i} \in \lnspan \Gamma\Gamma^{*} T\Gamma\Gamma^{*} \rnspan
  \subseteq \lnspan\Gamma\Gamma^{*}\Gamma A\Gamma^{*} \rnspan =
  \lnspan\Gamma A\Gamma^{*}\rnspan$  for all $i$.  Hence, $T \in
  \lnspan\Gamma A \Gamma^{*}\rnspan_{\Gamma}$.

  \smallskip

  ``$\supseteq$'': If $T \in \lnspan\Gamma A
  \Gamma^{*}\rnspan_{\Gamma}$, then $T\Gamma \subseteq
  \lnspan\Gamma A \Gamma^{*}\Gamma\rnspan = \lnspan\Gamma A
  \frakB\rnspan \subseteq \lnspan \Gamma A\rnspan$ and similarly
  $T^{*}\Gamma \subseteq \lnspan \Gamma A\rnspan$, whence $T \in
  \indgA$.
\end{proof}
\begin{lemma}
  $\indgA = \big\{ T \in \ind{\Gamma}{{\cal L}(H)} \,\big|\, \Gamma^{*} T
 \Gamma \subseteq A \big\}$.
\end{lemma}
\begin{proof}
  Clearly, $\indgA \subseteq \ind{\Gamma}{{\cal L}(H)}$. Let $T \in
  \ind{\Gamma}{{\cal L}(H)}$ such that $\Gamma^{*}T\Gamma \subseteq
  A$. Choose an approximate unit $(u_{i})_{i}$ for $\frakC$. Then for
  each $\gamma \in \Gamma$, the net $\big(u_{i}T\gamma\big)_{i}$
  converges in norm to $T\gamma$, and hence $T\Gamma \subseteq
  \lnspan\Gamma\Gamma^{*} T\Gamma \rnspan = \lnspan\Gamma A\rnspan$.
  Similarly, one shows that $T^{*}\Gamma \subseteq \lnspan \Gamma
  A\rnspan$.  Hence, $T \in \indgA$.
\end{proof}

\subsection{Definition and basic properties}
\label{subsection:fp-definition}

Throughout this paragraph, let $\cbasesb$ be a $C^{*}$-base.
We adopt the following terminology:
\begin{definition} \label{definition:fp-calgebra} 
Let $H$ be a Hilbert
  space, $\alpha\in \cfact(H;\cbasesb)$, and $A \subseteq {\cal L}(H)$
  a $C^{*}$-algebra. We call $A$ an {\em $\alpha$-module} and the
  triple $(H,A,\alpha)$ a {\em concrete $C^{*}$-$\cbasesb$-algebra} if
  $\rho_{\alpha}(\frakB^{\dagger})A \subseteq A$. If $A$ is
  nondegenerate, we call $(H,A,\alpha)$ nondegenerate.
 We put
  \begin{align*}
    \cfact(A;\cbasesb):=\{ \beta \in \cfact(H;\cbasesb) \mid A \text{
      is an } \beta\text{-module}\}
  \end{align*}
  and, if  $\alpha \in \cfact(A;\cbasesb)$ and $\cbasesc$ is a second
  $C^{*}$-base,
  \begin{align*}
   \cfact(A_{\alpha};\cbasesc):=\{ \beta \in \cfact(A;\cbasesb) \mid
    \beta \perp \alpha\}. 
  \end{align*}
  If $\alpha \in \cfact(A;\cbasesb)$ and $\beta \in
  \cfact(A_{\alpha};\cbasesc)$ (and $A$ is nondegenerate), we call
  $(H,A,\alpha,\beta)$ a {\em (nondegenerate) concrete
    $C^{*}$-$\cbasesb$-$\cbasesc$-algebra}.
\end{definition}

We define the fiber product of a concrete $C^{*}$-$\cbasesb$-algebra
$(H,A,\alpha)$ and a concrete $C^{*}$-$\cbaseosb$-algebra
$(K,B,\beta)$ by inducing $A$ and $B$ to $C^{*}$-algebras on
$\HfibreK$, using the subspaces
\begin{align*}
  \begin{aligned}
    \kbeta{1} &\subseteq {\cal L}(H,\HfibreK) &&\text{and} &\kalpha{1}
    &\subseteq {\cal L}(K,\HfibreK),
  \end{aligned}
\end{align*}
respectively, which were defined in Subsection
\ref{subsection:rtp-definition}.
\begin{definition} \label{definition:fp}  The {\em fiber product} of a
concrete  $C^{*}$-$\cbasesb$-algebra $(H,A,\alpha)$ and a concrete
  $C^{*}$-$\cbaseosb$-algebra $(K,B,\beta)$ is the $C^{*}$-algebra
  \begin{align*}
    \AfibreB := \ind{\kalpha{1}}{B} \cap \ind{\kbeta{2}}{A} \subseteq
    {\cal L}(\HfibreK).
  \end{align*}
\end{definition}

Unfortunately, we can not yet answer the following important question:
\begin{question}
  When is $\AfibreB \subseteq {\cal L}(\HfibreK)$ nondegenerate?
\end{question}

Note that if $(A,H,\alpha)$ is a concrete $C^{*}$-$\cbasesb$-algebra,
then $A' \subseteq \rho_{\alpha}(\frakB^{\dagger})'$.
\begin{lemma} \label{lemma:fp-commute} Let $(H,A,\alpha)$ be a
  concrete $C^{*}$-$\cbasesb$-algebra and $(K,B,\beta)$ a concrete
  $C^{*}$-$\cbaseosb$-algebra.  Then
  $(A' \rtensorh \Id) + (\Id  \rtensorh B') \subseteq (\AfibreB)'$.
\end{lemma}
\begin{proof}
  Let $T \in \AfibreB$ and $S \in A'$. Then for
  each $\eta\in \beta$, 
  \begin{align*}
    T(S \rtensorh \Id) |\eta\rangle_{\leg{2}} =
    T|\eta\rangle_{\leg{2}} S =      |\eta\rangle_{\leg{2}} T(S \rtensorh \Id)
  \end{align*}
  because $T|\eta\rangle_{\leg{2}} \in \lnspan
  \kbeta{2}A\rnspan$. Hence, $T(S \rtensorh \Id)=(S \rtensorh
  \Id)T$. A similar argument shows that $T(\Id \rtensorh R)=(\Id
  \rtensorh R)T$ for all $R \in B'$.
\end{proof}
\begin{lemma} \label{lemma:fp-stable} Let $(H,A,\alpha)$ be a concrete
  $C^{*}$-$\cbasesb$-algebra, $(K,B,\beta)$ a concrete
  $C^{*}$-$\cbaseosb$-algebra, and $\cbasesc$ a $C^{*}$-base. Then
  \begin{align*}
    \gamma \lt \beta &\in \cfact(\AfibreB;\cbasesc) \quad \text{for
      each }
    \gamma \in \cfact(A_{\alpha};\cbasesc), \\
    \alpha \rt \delta &\in \cfact(\AfibreB;\cbasesc)  \quad \text{for
      each }  \delta \in \cfact(B_{\beta};\cbasesc).
  \end{align*}
\end{lemma}
\begin{proof}
  We only prove the first assertion, the second one follows similarly:
  \begin{align*}
    \rho_{(\gamma \lt \beta)}(\frakC^{\dagger})(\AfibreB)\kalpha{1}
    &\subseteq \lnspan
    \rho_{\gamma}(\frakC^{\dagger})_{\leg{1}}\kalpha{1}B\rnspan \\ &=
    \lnspan |\rho_{\gamma}(\frakC^{\dagger}) \alpha\rangle_{\leg{1}}
    B\rnspan\subseteq \lnspan \kalpha{1}B\rnspan, \\
    \rho_{(\gamma \lt \beta)}(\frakC^{\dagger})(\AfibreB)\kbeta{2}
    &\subseteq \lnspan
    \rho_{\gamma}(\frakC^{\dagger})_{\leg{1}}\kbeta{2}A\rnspan \\
    &= \lnspan \kbeta{2} \rho_{\gamma}(\frakC^{\dagger}) A \rnspan
    \subseteq \lnspan \kbeta{2} A \rnspan. \qedhere
  \end{align*}
\end{proof}

The fiber product introduced above seems to fail to be associative and
to behave like the functor that associates to two $C^{*}$-algebras
$A$, $B$ the multiplier algebra $M(A \otimes B)$ instead of the tensor
product $A \otimes B$.

More precisely, let $\cbasesb$, $\cbasesc$ be $C^{*}$-bases,
$(H,A,\alpha)$ a concrete $C^{*}$-$\cbasesb$-algebra,
$(K,B,\beta,\gamma)$ a concrete
$C^{*}$-$\cbaseosb$-$\cbasesc$-algebra, and $(L,C,\delta)$ a concrete
$C^{*}$-$\cbaseosc$-algebra. Then we can form the fiber products
$\AfibreB$ and $B \fibre{\gamma}{\frakK}{\delta} C$, and by Lemma
\ref{lemma:fp-stable} also the following iterated fiber products:
\begin{align*}
(\AfibreB)   \fibre{\alpha \rt \gamma}{\frakK}{\delta} C \quad
\text{and} \quad
A \fibre{\alpha}{\frakH}{\beta \lt \delta} (B
\fibre{\gamma}{\frakK}{\delta} C).
\end{align*}
We expect that these $C^{*}$-algebras are {\em not} identified by the
canonical isomorphism
\begin{align*}
  (\HfibreK) \rtensor{\alpha \rt
  \gamma}{\frakK}{\delta} L \mycong H \rtensor{\alpha}{\frakH}{\beta
  \lt \delta} (K \rtensor{\gamma}{\frakK}{\delta} L)
\end{align*}
of Proposition \ref{proposition:rtp-associative}.
\begin{remark} \label{remark:fp-associative}
  Assume that we are given some  $n \geq 1$ and 
  \begin{itemize}
  \item $C^{*}$-bases $\cbases{B_{1}}{H_{1}}, \ldots,
    \cbases{B_{n}}{H_{n}}$,
  \item a concrete $C^{*}$-$\cbases{B_{1}}{H_{1}}$-algebra
    $(H_{1},A_{1},\alpha_{1})$,
  \item  concrete
    $C^{*}$-$\cbaseos{B_{i-1}}{H_{i-1}}$-$\cbases{B_{i}}{H_{i}}$-algebras
    $(H_{i},A_{i},\beta_{i-1},\alpha_{i})$ for $i=2,\ldots,n$, 
  \item a concrete  $C^{*}$-$\cbaseos{B_{n}}{H_{n}}$-algebra
    $(H_{n+1},A_{n+1},\beta_{n})$.
  \end{itemize}
  By Proposition\ref{proposition:rtp-associative}, we can form an
  iterated $C^{*}$-relative tensor product
  \begin{align} \label{eq:fp-associative-hs}
    (H_{1}) \rtensor{\alpha_{1}}{\frakH_{1}}{\beta_{1}}\cdots 
    \rtensor{\alpha_{n}}{\frakH_{n}}{\beta_{n}} (H_{n+1}),
  \end{align}
  and by Lemma \ref{lemma:fp-stable}, we can form form various
  iterated fiber products like
  \begin{align*}
    (\cdots (A_{1} \ast A_{2}) \ast A_{3}) \ast \cdots )\ast A_{n+1},\quad
    A_{1} \ast (\cdots \ast (A_{n-1} \ast (A_{n} \ast A_{n+1}) \cdots ),
  \end{align*}
  which can be identified with $C^{*}$-algebras on the Hilbert space
  \eqref{eq:fp-associative-hs}. Here, all possible ways (that we want
  to consider) of forming an iterated fiber product correspond
  bijectively to all possible ways of completely bracketing a
  product consisting of $n+1$ factors or, equivalently, with all
  binary trees with $n+1$ leaves. We denote by
  \begin{align*}
    (A_{1}) \fibre{\alpha_{1}}{\frakH_{1}}{\beta_{1}} \cdots
    \fibre{\alpha_{n}}{\frakH_{n}}{\beta_{n}} (A_{n+1}) \subseteq
    {\cal L}\big((H_{1})
    \rtensor{\alpha_{1}}{\frakH_{1}}{\beta_{1}}\cdots
    \rtensor{\alpha_{n}}{\frakH_{n}}{\beta_{n}} (H_{n+1})\big)
  \end{align*}
  the intersection of all  $C^{*}$-algebras obtained by
  iterating the fiber product construction in the ways described above.
\end{remark}

\subsection{Functoriality}
\label{subsection:fp-functorial}

We shall see that the fiber product is functorial with respect to
the following class morphisms:
\begin{definition} \label{definition:fp-morphism} Let $\cbasesb$ be a
  $C^{*}$-base and $(H,A,\alpha)$, $(K,B,\beta)$ concrete
  $C^{*}$-$\cbasesb$-algebras. A {\em morphism} from $(H,A,\alpha)$ to
  $(K,B,\beta)$ is a $*$-homomorphism $\pi \colon A \to B$ satisfying
  the following conditions:
 \begin{enumerate}
 \item $\pi(a\rho_{\alpha}(b^{\dagger})) =
   \pi(a)\rho_{\beta}(b^{\dagger})$ for all $a \in A$ and $b^{\dagger}
   \in \frakB^{\dagger}$,
 \item $\beta = [{\cal
     L}^{\pi}(H_{\alpha},K_{\beta})\alpha]$, where
   \begin{align*}
     {\cal L}^{\pi}(H_{\alpha}, K_{\beta}) := \big\{ V \in {\cal
       L}(H_{\alpha},K_{\beta}) \,\big|\, \forall a \in A: \pi(a)V = V
     a\big\}.
   \end{align*}
 \end{enumerate}
   We denote the set of all morphisms from $(H,A,\alpha)$ to $(K,B,\beta)$ by
   $\Mor\big(A_{\alpha},B_{\beta})$.
 \end{definition}
 \begin{remark} \label{remark:fp-morphism-nondegenerate}
   Let $\cbasesb$ be a $C^{*}$-base and $\pi$ a morphism of concrete
   $C^{*}$-$\cbasesb$-algebras  $(H,A,\alpha)$ and $(K,B,\beta)$. Then 
   \begin{align*}
     K=\lnspan \beta \frakH\rnspan = \lnspan {\cal
       L}^{\pi}(H_{\alpha},K_{\beta}) \alpha \frakH\rnspan = \lnspan  {\cal
       L}^{\pi}(H_{\alpha},K_{\beta}) H\rnspan,
   \end{align*}
   and if $A\subseteq {\cal L}(H)$ is nondegenerate, then also $\pi(A)
   \subseteq {\cal L}(K)$ is nondegenerate
   because
   \begin{align}
     \begin{aligned}
       \lnspan \pi(A)K\rnspan &= \lnspan \pi(A) {\cal
         L}^{\pi}(H_{\alpha},K_{\beta}) H\rnspan \\ &= \lnspan {\cal
         L}^{\pi}(H_{\alpha},K_{\beta}) A H\rnspan = \lnspan {\cal
         L}^{\pi}(H_{\alpha},K_{\beta}) H\rnspan = K.
     \end{aligned}
   \end{align}
 \end{remark}
 For each $C^{*}$-base $\cbasesb$, the class of all concrete
 $C^{*}$-$\cbasesb$-algebras together with the morphisms defined above
 evidently forms a category.
\begin{proposition} \label{proposition:fp-functorial}
  Let $\cbasesb$ be a $C^{*}$-base and 
  \begin{itemize}
  \item $\phi$  a morphism of nondegenerate concrete
    $C^{*}$-$\cbasesb$-algebras $(H,A,\alpha)$ and $(L,\gamma,C)$,
  \item $\psi$ a morphism of nondegenerate concrete
    $C^{*}$-$\cbaseosb$-algebras $(K,B,\beta)$ and $(M,\delta,D)$.
  \end{itemize}
 Then there
  exists a unique $*$-homomorphism
  \begin{gather} \nonumber
    \phi \ast \psi \colon \AfibreB \to C
    \fibre{\gamma}{\frakH}{\delta} D
\shortintertext{such that}
 \label{eq:fp-functorial}
    (\phi \ast \psi)(T) \cdot (X \rtensorh Y) = (X \rtensorh Y) \cdot T
  \end{gather}
whenever  $T \in \AfibreB$ and one of the following conditions holds:
i)  $X \in {\cal L}^{\phi}(H,L)$ and $Y \in {\cal
   L}^{\psi}(K_{\beta},D_{\delta})$ or
ii) $X \in
  {\cal L}^{\phi}(H_{\alpha},L_{\gamma})$ and $Y \in {\cal
    L}^{\psi}(K,D)$.
\end{proposition}
\begin{proof}
  Uniqueness follows from the fact that elements of the form $(X \rtensorh
  Y)\omega$, where $\omega \in \HfibreK$ and $X$ and $Y$ are as above
  are linearly dense in $\LfibreM$ by
  condition ii) in Definition \ref{definition:fp-morphism}.

  To prove existence, we first consider the special case where
  $(K,B,\beta)=(M,\delta,D)$ and $\psi = \Id_{B}$, and construct $\phi
  \ast \Id_{B}$. Then, a similar argument applies to the special case
  where $(H,A,\alpha)=(L,\gamma,C)$ and $\phi = \Id$, and proves the
  existence of $\Id_{C} \ast \psi$. Finally, we can put $\phi \ast
  \psi:=(\Id_{C} \ast \psi)\circ (\phi \ast \Id_{B})$.

 Consider the internal tensor product $E:=A {_{\rho_{\alpha}}\tl}
 \beta$. This is a $C^{*}$-module over $A$, and we have
 isomorphisms of internal tensor products
 \begin{align*}
   H {_{\Id}\tl} E &\cong H {_{\rho_{\alpha}} \tl} \beta \cong \HfibreK, &
   L {_{\phi} \tl} E &\cong L {_{\rho_{\gamma}} \tl}  \beta
   \cong L \rtensor{\gamma}{\frakH}{\beta} K.
 \end{align*}
 Moreover, we have $*$-homomorphisms
 \begin{align*}
   j_{H} &\colon {\cal L}_{A}(E) \to {\cal L}(H {_{\Id}\tl} E) \cong
   {\cal L}(\HfibreK), &  T &\mapsto \Id_{H} \tl T \equiv j_{H}(T), \\
   j_{L} &\colon {\cal L}_{A}(E) \to {\cal L}(L {_{\phi}\tl} E) \cong
   {\cal L}(L \rtensor{\gamma}{\frakH}{\beta} K), & T &\mapsto \Id_{L}
   \tl T \equiv j_{L}(T).
 \end{align*}
 Since $j_{H}$ is injective and $\AfibreB \subseteq \ind{\kbeta{2}}{A}
 \subseteq j_{H}({\cal L}_{A}(E))$ (Lemma \ref{lemma:ind-iso}), we can
 define $\phi \ast \Id \colon \AfibreB \to {\cal L}(L 
 \rtensor{\gamma}{\frakH}{\beta} K)$ to be the restriction of $j_{L}
 \circ j_{H}^{-1}$.
 
 Let $T \in j_{H}^{-1}(\AfibreB)$ and $X
 \in {\cal L}^{\phi}(H,L)$. Then
 the following diagram commutes and shows that
 $(\phi \ast\Id)(j_{H}(T)) \cdot (X \rtensorh \Id)
  = (X \rtensorh \Id) \cdot j_{H}(T)$:
\begin{align*}
  \xymatrix@C=15pt{
    {\HfibreK} \ar@{=}[r]
    \ar[d]^{X \rtensorh \Id} &
     {H {_{\Id}\tl} E} \ar[d]^{X \tl \Id}
     \ar[rr]^{j_{H}(T)\equiv\Id \tl T} &{\qquad\quad}&
     {H {_{\Id}\tl} E} \ar@{=}[r] \ar[d]^{X \tl
       \Id} &
     {\HfibreK} \ar[d]^{X \rtensorh
       \Id} \\
    {L \rtensor{\gamma}{\frakH}{\beta} K} \ar@{=}[r] &
     {L {_{\phi}\tl} E} 
     \ar[rr]^{j_{L}(T)\equiv\Id \tl T} &{\qquad\quad}&
     {L {_{\phi} \tl} E} \ar@{=}[r]  &
     {L \rtensor{\gamma}{\frakH}{\beta} K.}
   }
 \end{align*}
 Moreover, for each $S \in \AfibreB$ and $Y \in B'$, we have $S (\Id
 \rtensorh Y) = (\Id \rtensorh Y) S$ by Lemma
 \ref{lemma:fp-commute}. Summarizing, we find that condition
 \eqref{eq:fp-functorial} holds.

 Finally, let us show that $(\phi \ast \Id)(\AfibreB) \subseteq C
 \fibre{\gamma}{\frakH}{\beta} B$. Let $T \in \AfibreB$. Using the
 relation $\gamma=\lnspan {\cal L}^{\phi}(H,L)\alpha\rnspan$, Equation
 \eqref{eq:fp-functorial}, and the relation $T\kalpha{1}\subseteq
 \lnspan|\kalpha{1}B\rnspan$, we find
 \begin{align*}
   (\phi \ast \Id)(T) |\gamma\rangle_{\leg{1}} &\subseteq \lnspan
   (\phi \ast \Id)(T) |{\cal L}^{\phi}(H,L)\alpha\rangle_{\leg{1}}
   \rnspan  \\
   &= \lnspan ({\cal L}^{\phi}(H,L) \tr
   \Id) T |\alpha\rangle_{\leg{1}} \rnspan \\
   &\subseteq \lnspan ({\cal L}^{\phi}(H,L) \tr \Id)
   |\alpha\rangle_{\leg{1}} B \rnspan = \lnspan
   |\gamma\rangle_{\leg{1}} B \rnspan.
 \end{align*}
 Another application of Equation \eqref{eq:fp-functorial}  shows that
 $(\phi \ast \Id)(T)\kbeta{2} = T\kbeta{2} \subseteq \lnspan
 \kbeta{2}A\rnspan$.
\end{proof}
\begin{corollary}
  Let $\cbasesb$, $(H,A,\alpha)$, $(L,\gamma,C)$, $\phi$ and
  $(K,B,\beta)$, $(M,\delta,D)$, $\psi$ be as in Proposition
  \ref{proposition:fp-functorial}, and let $T \in \AfibreB$. Then
  \begin{align*}
    \langle \eta|_{\leg{2}} (\phi \ast \Id)(T) |\eta'\rangle_{\leg{2}}
    &= \phi\big(\langle\eta|_{\leg{2}}T|\eta'\rangle_{\leg{2}}\big)
    \quad \text{for all } \eta,\eta' \in \beta,
  \\    \langle \xi|_{\leg{1}} (\Id \ast \psi)(T) |\xi'\rangle_{\leg{1}}
    &= \psi\big(\langle\xi|_{\leg{1}}T|\xi'\rangle_{\leg{1}}\big)
    \quad \text{for all } \xi,\xi' \in \alpha.
  \end{align*}
\end{corollary}
\begin{proof}
  The equations can be deduced from the proof of
  Proposition \ref{proposition:fp-functorial}, but we give an alternative proof.
  Let $\eta,\eta' \in \beta$ and $X \in {\cal
    L}^{\phi}(H_{\alpha},L_{\gamma})$. Then
  \begin{align*}
    \langle\eta|_{\leg{2}}(\phi \ast \Id)(T) |\eta'\rangle_{\leg{2}} X
    = \langle\eta|_{\leg{2}} (X \rtensorh \Id)T
    |\eta'\rangle_{\leg{2}} = X \langle\eta|_{\leg{2}} T
    |\eta'\rangle_{\leg{2}}
  \end{align*}
  by \eqref{eq:fp-functorial}, and inserting $Xa =
  \phi(a)X$ for $a:=\langle \eta|_{\leg{2}}T|\eta'\rangle_{\leg{2}}
  \in A$, we find
  \begin{align*}
    \langle\eta|_{\leg{2}}(\phi \ast \Id)(T) |\eta'\rangle_{\leg{2}} X
    = \phi\big(\langle\eta|_{\leg{2}}T|\eta'\rangle_{\leg{2}}\big) X.
  \end{align*}
  Since $X \in {\cal L}^{\phi}(H_{\alpha},L_{\gamma})$ was arbitrary
  and $\lnspan {\cal L}^{\phi}(H_{\alpha},L_{\gamma}) \alpha\rnspan =
  \gamma$, the first equation of the corollary
  follows. The  second one follows similarly.
\end{proof}
\begin{theorem}
  Let $\cbasesb$, $\cbasesc$ be  $C^{*}$-bases, let
  \begin{itemize}
  \item $\phi$ be a morphism of nondegenerate concrete
    $C^{*}$-$\cbasesb$-algebras $(H,A,\alpha)$ and $(L,\gamma,C)$,
  \item $\psi$ be a morphism of nondegenerate concrete
    $C^{*}$-$\cbaseosb$-algebras $(K,B,\beta)$ and $(M,\delta,D)$,
  \end{itemize}
  and assume that $\AfibreB \subseteq {\cal L}(\HfibreK)$ is
  nondegenerate. 
  \begin{enumerate}
  \item If $\alpha' \in \cfact(A_{\alpha};\cbasesc)$, $\gamma' \in
    \cfact(C_{\gamma};\cbasesc)$,  $\phi \in
    \Mor(A_{\alpha'}, C_{\gamma'})$, then
    $\phi \ast \psi \in \Mor \big((\AfibreB)_{(\alpha' \lt \beta)},
      (C \fibre{\gamma}{\frakH}{\delta} D)_{(\gamma' \lt \delta)}\big)$.
  \item  If $\beta' \in \cfact(B_{\beta};\cbasesc)$, $\delta' \in
    \cfact(D_{\delta};\cbasesc)$,  $\psi \in
    \Mor(B_{\beta'}, D_{\delta'})$, then
    $\phi \ast \psi \in \Mor \big((\AfibreB)_{(\alpha \rt \beta')},
      (C \fibre{\gamma}{\frakH}{\delta} D)_{(\gamma \rt \delta')}\big)$.
  \end{enumerate}
\end{theorem}
\begin{proof}
  We only prove i); assertion ii) follows similarly. 

  By Lemma \ref{lemma:fp-stable}, $(\HfibreK,\AfibreB,\alpha' \lt
  \beta)$ is a concrete $C^{*}$-$\cbasesc$-algebra.

  We show that $\phi \ast \psi$ satisfies condition i) in Definition
  \ref{definition:fp-morphism}. Fix $T \in \AfibreB$ and $c^{\dagger}
  \in \frakC^{\dagger}$, and let $X \in {\cal
    L}^{\phi}(H_{\alpha},L_{\gamma})$  $Y \in {\cal
    L}^{\psi}(K_{\beta},M_{\delta})$. By \eqref{eq:fp-functorial},
  \begin{align} \label{eq:fp-functorial-equivariant}
    \begin{aligned}
      (\phi \ast \psi)(T\rho_{(\alpha' \lt \beta)}(c^{\dagger})) (X
      \rtensorh Y) &=
      (X \rtensorh Y) (T\rho_{(\alpha' \lt \beta)}(c^{\dagger})) \\
      &= (\phi \ast \psi)(T)(X \rtensorh Y) \rho_{(\alpha' \lt
        \beta)}(c^{\dagger}).
    \end{aligned}
  \end{align}
  For each $a \in A$, we have $\rho_{\alpha'}(c^{\dagger})a \in A$ and
  \begin{align*}
    X\rho_{\alpha'}(c^{\dagger})a =
    \phi\big(\rho_{\alpha'}(c^{\dagger})a\big)X =
    \rho_{\gamma'}(c^{\dagger})\phi(a)X= \rho_{\gamma'}(c^{\dagger})X
    a
  \end{align*}
  because $\phi \in \Mor(A_{\alpha'},C_{\gamma'})$ and $X \in {\cal
    L}^{\phi}(H_{\alpha},L_{\gamma})$.  Since $A \subseteq {\cal
    L}(H)$ is nondegenerate, we can conclude $
  X\rho_{\alpha'}(c^{\dagger})=\rho_{\gamma'}(c^{\dagger})X$. Inserting
  this equation and the relations $\rho_{(\alpha' \lt \beta)} =
  (\rho_{\alpha'})_{\leg{1}}$, $\rho_{(\gamma' \lt \delta)} =
  (\rho_{\gamma'})_{\leg{1}}$ into
  \eqref{eq:fp-functorial-equivariant}, we find
  \begin{align*}
      (\phi \ast \psi)(T\rho_{(\alpha' \lt \beta)}(c^{\dagger})) (X
      \rtensorh Y) &=
       (\phi \ast \psi)(T) \rho_{(\gamma' \lt
        \beta)}(c^{\dagger})(X \rtensorh Y).
  \end{align*}
  Since $X \in {\cal L}^{\phi}(H_{\alpha},L_{\gamma})$ and $Y\in {\cal
    L}^{\psi}(K_{\beta},M_{\delta})$ were arbitrary, we can conclude
  \begin{align*}
    (\phi \ast \psi)(T\rho_{(\alpha' \lt \beta)}(c^{\dagger})) =(\phi
    \ast \psi)(T) \rho_{(\gamma' \lt \beta)}(c^{\dagger}).
  \end{align*}
  
  To show that $\phi \ast \psi$ satisfies condition ii) in Definition
  \ref{definition:fp-morphism}.  Let $X \in {\cal
    L}^{\phi}(H_{\alpha'},L_{\gamma'})$ and $Y \in {\cal
    L}^{\psi}(K_{\beta},M_{\delta})$. Then a similar argument as above
  shows that $X\rho_{\alpha}(b^{\dagger})=\rho_{\gamma}(b^{\dagger})$
  for all $b^{\dagger} \in \frakB^{\dagger}$, so that $X \rtensorh Y
  \in {\cal L}(\HfibreK, \LfibreM)$ is well-defined. By Equation
  \eqref{eq:fp-functorial}, $X \rtensorh Y \in {\cal L}^{\phi \ast
    \psi}\big(\HfibreK, \LfibreM\big)$, and since
    \begin{align} \label{eq:fp-aux}
      (X \rtensorh Y)\kbeta{2}\alpha' = |Y\beta\rangle_{\leg{2}}
      (X\alpha') \subseteq |\delta\rangle_{\leg{2}} \gamma' \subseteq
      \gamma' \lt \delta,
    \end{align}
    $X \rtensorh Y \in {\cal L}^{\phi \ast
      \psi}\big((\HfibreK)_{\alpha' \lt \beta}, (L
    \rtensor{\gamma}{\frakH}{\delta} M)_{\gamma' \lt
      \delta}\big)$. This inclusion, equation \eqref{eq:fp-aux} and
    the relations $\lnspan {\cal
      L}^{\phi}(H_{\alpha'},L_{\gamma'})\alpha'\rnspan=\gamma'$ and
    $\lnspan {\cal L}^{\psi}(K_{\beta},M_{\delta})\beta\rnspan =
    \delta$ imply
  \begin{align*}
    \gamma' \lt \delta = \lnspan
    |\delta\rangle_{\leg{2}}\gamma'\rnspan &= \lnspan |{\cal
      L}^{\psi}(K_{\beta},M_{\delta}) \beta\rangle_{\leg{2}} {\cal
      L}^{\phi}(H_{\alpha'},L_{\gamma'}) \alpha'\rnspan \\ &\subseteq
    \lnspan {\cal L}^{\phi \ast \psi}\big((\HfibreK)_{\alpha' \lt
      \beta}, (\LfibreM)_{\gamma' \lt \delta}\big) (\alpha' \lt \beta)
    \rnspan. \qedhere
\end{align*}
\end{proof}

\begin{definition}
  A {\em concrete Hopf $C^{*}$-bimodule} is a tuple consisting of a
  $C^{*}$-base $\cbasesb$, a nondegenerate concrete
  $C^{*}$-$\cbasesb$-$\cbaseosb$-algebra $(H,A,\alpha,\beta)$, and a
  $*$-homomorphism $\Delta \colon A \to \AfibreA$ subject to the
  following conditions:
  \begin{enumerate}
  \item $\AfibreA \subseteq {\cal L}(H \rtensorab H)$ is nondegenerate,
  \item $\Delta \in \Mor\big(A_{\alpha}, (\AfibreA)_{\alpha \rt \alpha}\big)
    \cap \Mor\big(A_{\beta}, (\AfibreA)_{\beta \lt \beta}\big)$, and
  \item the following diagram commutes:
    \begin{align*}
      \xymatrix@C=10pt@R=10pt{
        A \ar[rrr]^{\Delta} \ar[ddd]^{\Delta} &&& {\AfibreA} \ar[dd]^{\Id
          \ast \Delta} \\ \\
        &&& {A \fibre{\alpha}{\frakH}{\beta \lt \beta} (\AfibreA)}
        \ar@{^(->}[d] \\
        {\AfibreA} \ar[rr]^(0.35){\Id \ast \Delta} &&
        {(\AfibreA) \fibre{\alpha \rt \alpha}{\frakH}{\alpha} A}
        \ar@{^(->}[r] &{\cal L}(H \fibreab H \fibreab H).
      }
    \end{align*}
  \end{enumerate}
\end{definition}
\begin{remark}
  If $(\cbasesb,H,\alpha,\beta,\Delta)$ is a concrete Hopf
  $C^{*}$-bimodule, then $(\Delta \ast \Id)(\Delta(A))$ and $(\Id\ast
  \Delta)(\Delta(A))$ are contained in $A \fibreab A \fibreab A$ (see
  Remark \ref{remark:fp-associative}), and for each $n \geq 2$, one
  can define a $*$-homomorphism
  \begin{align*}
    \Delta^{(n)} \colon A \to \underbrace{A \fibreab \cdots \fibreab
      A}_{(n+1) \text{ factors}}
  \end{align*}
  by iterated applications of $\Delta$, where the precise way in which
  the applications of $\Delta$ are iterated does not matter. Note that
  for each $n \geq 2$, all those possible ways correspond bijectively to
  all binary trees that have $n+1$ leaves.
\end{remark}
\section{$C^{*}$-pseudo-multiplicative unitaries}
\label{section:pmu}

In this section, we introduce $C^{*}$-pseudo-multiplicative unitaries
and explain their relation to several other generalizations of
multiplicative unitaries.  Following the treatment of multiplicative
unitaries given by Baaj and Skandalis \cite{baaj:2}, we define the legs
of such a $C^{*}$-pseudo-multiplicative unitary and show that under a
suitable regularity condition, these legs form concrete Hopf
$C^{*}$-bimodules.  Throughout, we apply the concepts and techniques
developed in the preceding sections.

\subsection{Definition}
\label{subsection:pmu-definition}

 Recall that a multiplicative unitary \cite[D\'efinition 1.1]{baaj:2}
 on a Hilbert space $H$ is a unitary $V \colon H \otimes H \to H
 \otimes H$ that satisfies the so-called pentagon equation
 $V_{12}V_{13}V_{23}=V_{23}V_{12}$.  Here, $V_{12},V_{13},V_{23}$ are
 operators on $H \otimes H \otimes H$, defined by $V_{12} = V \otimes
 \Id$, $V_{23} = \Id \otimes V$, $V_{13} = (\Sigma \otimes
 \Id)V_{23}(\Sigma \otimes \Id) = (\Id \otimes \Sigma) V_{12}(\Id
 \otimes \Sigma)$, where $\Sigma \in {\cal B}(H \otimes H)$ denotes
 the flip $\eta \otimes \xi \mapsto \xi \otimes \eta$.  We extend this
 concept, replacing the ordinary tensor product of Hilbert spaces by
 the $C^{*}$-relative tensor product as follows.

 Let $H$ be a Hilbert space, $\cbasesb$ a $C^{*}$-base, and
 \begin{align*}
  \alpha &\in \cfact(H;\cbasesb), & \hbeta &\in \cfact(H;\cbaseosb), & \beta
  &\in \cfact(H;\cbaseosb)
\end{align*}
pairwise compatible $C^{*}$-factorizations.
\begin{lemma} \label{lemma:pmu-pentagon}
  Let $V \in {\cal L}\big(\Hsource, \Hrange\big)$ and assume that
  \begin{gather} \label{eq:pmu-intertwine}
    \begin{aligned}
      V_{*}(\alpha \lt \alpha) &= \alpha \rt \alpha, &
      V_{*}(\hbeta \rt \beta) &= \hbeta \lt \beta, \\
      V_{*}(\hbeta \rt \hbeta) &= \alpha \rt \hbeta, & V_{*}(\beta \lt
      \alpha) &= \beta \lt \beta.
    \end{aligned}
  \end{gather}
  Then all operators in the following diagram are well-defined,
  \begin{gather} \label{eq:pmu-pentagon}
    \begin{gathered}
      \xymatrix@C=-20pt@R=10pt{ && {\hspace{-5pt} \Htwo \hspace{-5pt}}
        \ar[rrd]-<10pt,-8pt>^(0.7){ \Id \rtensorh V}
        \ar@{<-}[lld]-<-10pt,-8pt>_(0.7){ V \rtensorh \Id}
        && \\
        {\Hone \hspace{-5ex}} \ar[rdd]^(0.6){ \Id \rtensorh V} &&&&
        {\hspace{-5ex} \Hthree,}
        \\
        \\
        &{\Hfive} \ar[dd]^(0.5){ \Id\rtensorh \Sigma} & & {\Hfour}
        \ar[ruu]^(0.4){ V \rtensorh \Id}
        &   \\ \\
        &{\Hfourlt} \ar@<1.5pt>[rr]^{ V \rtensorh \Id} & & {\Hfourrt}
        \ar[uu]^(0.5){  \Sigma_{\leg{23}}} & }
    \end{gathered}
  \end{gather}
where $\Sigma_{\leg{23}}$ denotes the isomorphism
\begin{align*}
  \Hfourrt \cong (H {_{\rho_{\alpha}} \tl} \beta)  {_{\rho_{\hbeta \lt
      \beta}} \tl} \alpha &\mycong (H {_{\rho_{\hbeta}} \tl} \alpha)
{_{\rho_{\alpha \lt \alpha}} \tl} \beta \cong \Hfour, \\
(\zeta \tl \xi) \tl \eta &\mapsto
  (\zeta \tl \eta) \tl \xi.
\end{align*}
\end{lemma}
\begin{proof}
  This follows easily from the functoriality of the relative tensor
  product \ref{proposition:rtp-functorial}; for example, we have operators
  \begin{gather*}
    H \sfsource H \sfsource H \cong \big(\sHsource\big)
    \htensor{\hbeta \rt \hbeta}{\alpha} H \xrightarrow{V \rtensorh \Id}
    \big(\sHrange\big) \htensor{\alpha \rt \hbeta}{\alpha} H \cong H
    \sfrange H \sfsource H, \\
    H \htensor{\hbeta}{\alpha \rt \alpha} \big(\sHrange\big)
    \xrightarrow{\Id \rtensorh \Sigma} H \htensor{\hbeta}{\alpha \lt
      \alpha} \big(H \htensor{\beta}{\alpha} H\big) \cong H \sfsource
    H \htensor{\beta}{\alpha} H. \qedhere
  \end{gather*}
\end{proof}

\begin{definition}  \label{definition:pmu}
  Let $H$ be a Hilbert space, $\cbasesb$ a $C^{*}$-base, and $\alpha
  \in \cfact(H;\cbasesb)$, $\hbeta,\beta \in \cfact(H;\cbaseosb)$
  pairwise compatible.  A unitary $V\in {\cal L}(\Hsource,\Hrange)$ is
  {\em $C^{*}$-pseudo-multiplicative} if Equations
  \eqref{eq:pmu-intertwine} hold and Diagram \eqref{eq:pmu-pentagon}
  commutes.
\end{definition}
\begin{remark} \label{remark:pmu-opposite}
  If $V \in {\cal L}(\Hsource,\Hrange)$ is a
  $C^{*}$-pseudo-multiplicative unitary, then also
  \begin{align*}
    V^{op} :=\Sigma V^{*} \Sigma \colon H \htensor{\beta}{\alpha} H
    \xrightarrow{\Sigma} \Hrange \xrightarrow{V^{*}} \Hsource
    \xrightarrow{\Sigma} H \htensor{\alpha}{\hbeta} H
  \end{align*}
  is a $C^{*}$-pseudo-multiplicative unitary; here, the r\^oles of
  $\hbeta$ and $\beta$ get reversed.
\end{remark}
This definition subsumes the following special cases:
\begin{enumerate}
\item The representations $\rho_{\hbeta}$ and $\rho_{\beta}$ restrict
  to representations $\rho_{\hbeta}\colon \frakB \to {\cal L}_{\frakB}(\alpha)$
  and $\rho_{\beta}\colon \frakB \to {\cal L}_{\frakB}(\alpha)$, and the unitary
  $V$ restricts to a unitary operator on right $C^{*}$-modules
  \begin{align} \label{eq:pmu-cpmu} 
   V_{\alpha} \colon \alpha {_{\rho_{\hbeta}} \tl} \alpha \to \alpha
  \tr_{\rho_{\beta}} \alpha
\end{align}
which is a pseudo-multiplicative unitary on $C^{*}$-modules in the
sense of Timmermann \cite{timmermann}.
\item If $\hbeta=\alpha$, then $\frakB$ is commutative by Remark
  \ref{remark:rtp-commute}, and the restriction $V_{\alpha}$ is a
  pseudo-multiplicative unitary in the sense of O'uchi \cite{ouchi}.
  Similarly, if $\beta=\alpha$, then $V_{\alpha}^{op}$ is a
  pseudo-multiplicative unitary in the sense of \cite{ouchi}.
\item If $\alpha=\hbeta=\beta$, then again $\frakB$ is commutative, and
  $V_{\alpha}$ is a continuous field of multiplicative unitaries in
  the sense of Blanchard \cite{blanchard}.
\item Assume that $\cbasesb$ is the $C^{*}$-base associated to a
  proper KMS-weight $\mu$ on a $C^{*}$-algebra $B$, that is, $\frakH =
  H_{\mu}$ and $\frakB=\pi_{\mu}(B)$,
  $\frakB^{\dagger}=\pi_{\mu^{op}}(B^{op})$ (see Example
  \ref{example:rtp-base-weight}).  Put $N:=\frakB'' \subset {\cal
    L}(H_{\mu})$, denote by $\tilde \mu$ the extension of $\mu$ to a
  normal semifinite faithful weight on $N$, and denote by $\tilde
  \rho_{\alpha}$, $\tilde \rho_{\hbeta}$, and $\tilde \rho_{\beta}$
  the unique normal extensions of $\rho_{\alpha}$, $\rho_{\hbeta}$,
  and $\rho_{\beta}$, respectively, to $N$ or $N^{op}$ (see Subsection
  \ref{subsection:rtp-comparison}).  Then we have isomorphisms
  \begin{align*}
    \Hsource &\cong H
    \vtensor{\tilde\rho_{\hbeta}}{\tilde\mu}{\tilde\rho_{\alpha}} H, &
    \Hrange &\cong H
    \vtensor{\tilde\rho_{\beta}}{\tilde\mu^{op}}{\tilde\rho_{\alpha}}
    H,
  \end{align*}
  and with respect to these isomorphisms, $V$ is a
  pseudo-multiplicative unitary on Hilbert spaces in the sense of
  Vallin \cite{vallin:2}.
\end{enumerate}

\subsection{The legs of a $C^{*}$-pseudo-multiplicative unitary}
\label{subsection:pmu-legs}

To every multiplicative unitary $V$, Baaj and Skandalis associate two
algebras $\widehat{A}(V)$, $A(V)$ and two normal $*$-homomorphisms
$\hDelta_{V}, \Delta_{V}$ such that, under favorable circumstances
like regularity, $(\widehat{A}(V),\hDelta_{V})$ and
$(A(V),\Delta_{V})$ are Hopf $C^{*}$-algebras \cite{baaj:2}. 

Their construction carries over to $C^{*}$-pseudo-multiplicative
unitaries as follows.
Let $V \in {\cal L}(\Hsource,\Hrange)$ a pseudo-multiplicative
unitary, where $H$, $\cbasesb$, and $\alpha,\hbeta,\beta$ are as in
Definition \ref{definition:pmu}.

\paragraph{The algebras $\hA(V)$ and $A(V)$}
We define spaces $\hA(V) \subseteq {\cal L}(H)$ and $A(V) \subseteq
{\cal L}(H)$, using the spaces of ket-bra operators $\kalpha{2},
\khbeta{1} \subseteq {\cal L}(H, \Hsource)$ and $\bbeta{2},\balpha{1}
\subseteq {\cal L}(\Hrange, H)$ introduced in Subsection
\ref{subsection:rtp-definition}, as follows:
\begin{align*}
\hA:=  \hA(V) &:= \lnspan \bbeta{2} V \kalpha{2}\rnspan \subseteq {\cal L}(H),
  &
A:=  A(V) &:= \lnspan\Gamma^{*} V \khbeta{1}\rnspan \subseteq {\cal L}(H).
\end{align*}

The definition of $\hA(V)$ and $A(V)$ is symmetric in the following sense:
\begin{lemma} \label{lemma:pmu-legs-symmetry}
 $\hA(V^{op})=A(V)^{*}$ and  $A(V^{op})=\hA(V)^{*}$.
\end{lemma}
\begin{proof}
  Switching from $V$ to $V^{op}$, the r\^oles of $\beta$ and $\hbeta$
  get reversed, so
  \begin{align*}
    \hA(V^{op}) = \lnspan \bhbeta{2} \Sigma V^{*} \Sigma \kalpha{2}\rnspan =
    \lnspan \bhbeta{1} V^{*}\kalpha{1}\rnspan = \lnspan
   \Gamma^{*}V\khbeta{1}\rnspan^{*} = A(V)^{*},
  \end{align*}
  and similarly $A(V^{op}))=\hA(V)^{*}$.
\end{proof}
The spaces $\hA(V)$ and $A(V)$  are related to the representations
$\rho_{\alpha},\rho_{\hbeta},\rho_{\beta}$ as follows:
\begin{lemma} \label{lemma:pmu-legs-modules}
  We have 
  \begin{gather*}
    \lnspan \hA \rho_{\hbeta}(B)\rnspan = \lnspan \rho_{\hbeta}(B)
    \hA\rnspan = \lnspan \hA \rho_{\alpha}(B^{op})\rnspan = \lnspan
    \rho_{\alpha}(B^{op}) \hA\rnspan = \hA \subseteq {\cal
      L}(H_{\beta}) \shortintertext{and} \lnspan A
    \rho_{\beta}(B)\rnspan = \lnspan \rho_{\beta}(B) A\rnspan =
    \lnspan A \rho_{\alpha}(B^{op})\rnspan = \lnspan
    \rho_{\alpha}(B^{op}) A\rnspan = A \subseteq {\cal L}(H_{\hbeta}).
  \end{gather*}
\end{lemma}
\begin{proof}
  We only prove the assertions concerning $\hA$: Using  Equation
  \eqref{eq:pmu-intertwine} and the relations $\lnspan \alpha B\rnspan =
  \alpha = \lnspan \rho_{\beta}(B)\alpha\rnspan$ and $\lnspan
  \rho_{\alpha}(B^{op})\beta\rnspan = \beta = \lnspan \beta
  B^{op}\rnspan$, we find
  \begin{align*}
    \lnspan \bbeta{2} V \kalpha{2} \rho_{\hbeta}(B)\rnspan &=
    \lnspan \bbeta{2} V |\alpha B\rangle_{2} \rnspan = \hA, \\
    \lnspan \rho_{\hbeta}(B) \bbeta{2} V \kalpha{2} \rnspan &=
    \lnspan \bbeta{2} \rho_{\hbeta \lt \beta}(B) V \kalpha{2} \rnspan \\
    &= \lnspan \bbeta{2} V \rho_{\hbeta \rt \beta}(B) \kalpha{2}
    \rnspan = \lnspan \bbeta{2} V
    |\rho_{\beta}(B)\alpha\rangle_{2}\rnspan =
    \hA, \\
    \lnspan \bbeta{2} V \kalpha{2} \rho_{\alpha}(B^{op}) \rnspan &=
    \lnspan \bbeta{2} V \rho_{\alpha\lt \alpha}(B^{op}) \kalpha{2}
    \rnspan \\ &= \lnspan \bbeta{2} \rho_{\alpha\rt \alpha}(B^{op}) V
    \kalpha{2} \rnspan = \lnspan \langle
    \rho_{\alpha}(B^{op})\beta|_{\leg{2}} V \kalpha{2} \rnspan = \hA,
    \\
    \lnspan \rho_{\alpha}(B^{op})  \bbeta{2} V \kalpha{2}
    \rnspan&= \lnspan \langle\beta B^{op}|_{\leg{2}} V \kalpha{2} = \hA.
  \end{align*}
  The inclusion $\hA \subseteq {\cal L}(H_{\beta})$ follows from the
  inclusions $\kalpha{2} \subseteq {\cal L}\big(H_{\beta},
  (\Hsource)_{\beta \lt \alpha}\big)$, $ V \in {\cal
    L}\big((\Hsource)_{\beta \lt \alpha}, (\Hrange)_{\beta \lt
    \beta}\big)$, and $ \bbeta{2} \subseteq {\cal
    L}\big((\Hrange)_{\beta \lt \beta}, H_{\beta} \big)$.
\end{proof}

The spaces $\hA$ and $A$ are algebras:
\begin{proposition} \label{proposition:pmu-legs-algebras}
  $\lnspan \hA \hA\rnspan = \hA$ and $\lnspan A A \rnspan = A$.
\end{proposition}
\begin{proof}
  We only prove the first equation. The following diagram commutes and
  shows that $\lnspan \hA \hA\rnspan = \lnspan \bbeta{2}\balpha{3}
  \Vl{12}\kalpha{3}\kalpha{2}\rnspan$:
  \begin{gather*} \hspace{-0.25cm} \smalldiagram
    \xymatrix@C=-3pt@R=25pt{ & H \ar[d]^{\kalpha{2}} \ar[rr]^{ \hA}
      \ar `l/0pt[l] [ld]^{\kalpha{2}} \ar@{}[rd]|{\scriptstyle \qquad \ (D)} & & H
      \ar@{}[dd]|(0.4){\scriptstyle(C)} \ar[rd]^(0.65){\kalpha{2}}
      \ar[rr]^{\hA} & \ar@{}[rd]|{\scriptstyle (D) \qquad \ } & H  &  \\
      {\sHsource} \ar `d[dd]^(0.75){\kalpha{3}} [ddr] & {\sHsource}
      \ar[r]^(0.55){V} \ar[d]^{\kalpha{2}} & {\sHrange}
      \ar[ru]^(0.35){\bbeta{2}} \ar[rd]^{ \kalpha{2}} & & {\sHsource}
      \ar[r]^(0.45){V} & {\sHrange} \ar[u]^{\bbeta{2}} & {\sHrange}
      \ar `u[u]^{\bbeta{2}} [ul] \\
      & {\Hfive} \ar[rr]^{\Vl{13}} \ar[d]^{\Vl{23}^{*}} && {\Hfour}
      \ar[ru]^{\bbeta{3}} \ar[rr]^{\Vl{12}}
      \ar@{}[d]|(0.4){\scriptstyle (P)} &&
      {\Hthree}
      \ar[u]^{\bbeta{3}}   & \\
      & {\Hone} \ar[rrrr]^{ \Vl{12}} & &&& {\Htwo} \ar[u]^{\Vl{23}}
      \ar `r/0pt[r] [ruu]^(0.25){\balpha{3}} & }
      \end{gather*}
      Indeed, the cells labeled by (D) commute by definition, cell
      (C) commutes because
      \begin{gather} \label{eq:pmu-diagram-commutes}
          |\xi\rangle_{\leg{2}} \langle \eta'|_{\leg{2}} (\zeta \tl
          \eta) = \rho_{\alpha}\big(\eta'{}^{*}\eta\big)\zeta \tl \xi
           = \rho_{\alpha \lt \alpha}\big(\eta'{}^{*}\eta\big) (\zeta
          \tl \xi) = \langle \eta'|_{\leg{3}} |\xi\rangle_{\leg{2}}
          (\zeta \tl \eta)
      \end{gather} 
      for all $\xi \in \alpha$, $\eta,\eta'\in \beta$, $\zeta \in H$,
      cell (P) is just Diagram \eqref{eq:pmu-pentagon}, and the
      remaining cells commute because of
      \eqref{eq:pmu-intertwine}. 

       The following commutative diagram   shows that $\lnspan
      \bbeta{2}\balpha{3}\Vl{12}\kalpha{3}\kalpha{2}\rnspan = \hA$ and
      completes the proof:
       \begin{gather*}  \smalldiagram
        \xymatrix@C=40pt@R=25pt{ & H \ar `l/0pt[l]
          [ld]^{{\kalpha{2}}} \ar[r]^{\hA}
          \ar[d]^{{\kalpha{2}}} & H \\
          {\sHsource} \ar `d/0pt[d]^{\kalpha{3}}
          [dr] \ar[r]^(0.55){\rho_{\hbeta \rt \hbeta}(B)} &
          {\sHsource} \ar[r]^{
            V} & {\sHrange} \ar[u]^{\bbeta{2}} \\
          & {\Hone} \ar[u]^{
            \balpha{3}} \ar[r]^{\Vl{12}} & {\Htwo} \ar[u]^{\balpha{3}} } 
      \end{gather*}
\end{proof}

The algebras $\hA$ and $A$ are nondegenerate in the following strong
sense:
\begin{proposition}
  $\lnspan \hA \beta\rnspan = \beta = \lnspan \hA^{*} \beta\rnspan$
  and $\lnspan A \hbeta\rnspan = \hbeta = \lnspan A^{*}\hbeta\rnspan$.
\end{proposition}
\begin{proof}
  We only prove the first equation, the others follow similarly: Since
  $V_{*}(\beta \lt \alpha) = \beta \lt \beta$, 
  \begin{gather*}
    \lnspan \hA \beta\rnspan = \lnspan \bbeta{2} V
    \kalpha{2}\beta\rnspan = \lnspan \bbeta{2} \kbeta{2}\beta\rnspan =
    \lnspan \rho_{\alpha}(B^{op})\beta\rnspan =\beta. \qedhere
  \end{gather*}
\end{proof}
\begin{corollary}
  $\lnspan \hA H\rnspan = H = \lnspan \hA^{*} H\rnspan$ and $\lnspan A H
  \rnspan = H = \lnspan A^{*}H\rnspan$. \qed
\end{corollary}

\paragraph{The comultiplications $\hDelta_{V}$ and $\Delta_{V}$}
We define maps
\begin{gather*}
  \widehat{\Delta}= \widehat{\Delta}_{V} \colon \rho_{\beta}(\frakB)' \to
  {\cal L}\big(\Hsource\big), \quad y \mapsto V^{*}(1 \rtensorh
  y)V \shortintertext{and} \Delta = \Delta_{V} \colon
  \rho_{\hbeta}(\frakB)' \to {\cal L}\big(\Hrange), \quad z
  \mapsto V(z \rtensorh 1)V^{*}.
\end{gather*}
\begin{proposition} \label{proposition:pmu-delta}
 We have
  \begin{align*}
    \hDelta \in \Mor\Big(\rho_{\beta}(\frakB)'_{\alpha}, {\cal
      L}(\Hsource)_{\alpha \lt \alpha}\Big) \cap
    \Mor\Big(\rho_{\beta}(\frakB)'_{\hbeta}, {\cal
      L}(\Hsource)_{\hbeta \rt \hbeta}\Big) \shortintertext{and}
    \Delta \in \Mor\Big(\rho_{\hbeta}(\frakB)'_{\beta}, {\cal
      L}(\Hrange)_{\beta \lt \beta}\Big) \cap
    \Mor\Big(\rho_{\hbeta}(\frakB)'_{\beta}, {\cal
      L}(\Hrange)_{\beta \rt \beta}\Big).
  \end{align*}
\end{proposition}
\begin{proof}
  We only prove the assertion concerning $\hDelta$.
  The intertwining relations \eqref{eq:pmu-intertwine} immediately
  imply that for each $b \in \frakB$ and $b^{\dagger} \in
  \frakB^{\dagger}$,
  \begin{align*}
    \hDelta(\rho_{\alpha}(b^{\dagger})) &= \rho_{\alpha \lt
      \alpha}(b^{\dagger}), &
    \hDelta(\rho_{\hbeta})(b) &= \rho_{\hbeta \rt \hbeta}(b), &
    \Delta(\rho_{\beta}(b)) &= \rho_{\beta \lt \beta}(b), &
    \Delta(\rho_{\alpha}(b^{\dagger})) &= \rho_{\alpha \rt
      \alpha}(b^{\dagger}). 
  \end{align*}
  We claim that $V^{*}\kalpha{1} \subseteq {\cal
      L}^{\hDelta}\big(H_{\alpha},(\Hsource)_{}\big)$. Indeed, for
    each $y \in \rho_{\beta}(\frakB)'$ and $\xi \in \alpha$,
    \begin{align*}
      \hDelta(y)V^{*}|\xi\rangle_{\leg{1}} = V^{*}(1 \rtensorh
      y)VV^{*} |\xi\rangle_{\leg{1}} = V^{*}(1 \rtensorh
      y)|\xi\rangle_{\leg{1}} = V^{*}|\xi\rangle_{\leg{1}} y.
    \end{align*}
    Therefore, $\alpha \lt \alpha = V^{*} \lnspan \kalpha{1} \alpha
    \rnspan \subseteq \lnspan {\cal
      L}^{\hDelta}\big(H_{\alpha},(\Hsource)_{\alpha \lt \alpha}\big)
    \alpha\rnspan$, and a similar argument shows that $\hbeta
    \rt\hbeta \subseteq \lnspan {\cal
      L}^{\hDelta}\big(H_{\hbeta},(\Hsource)_{\hbeta \rt \hbeta}\big)
    \hbeta\rnspan$. The assertion follows.
\end{proof}

\subsection{Regularity}
\label{subsection:pmu-regular}

The regularity condition for multiplicative unitaries, introduced by
Baaj and Skandalis \cite{baaj:2} and generalized to
pseudo-multiplicative unitaries by Enock \cite{enock:10}, carries over to
$C^{*}$-pseudo-multiplicative unitaries as follows:
\begin{definition}
  Let $V \in {\cal L}(\Hsource,\Hrange)$ a pseudo-multiplicative
  unitary, where $H$, $\cbasesb$, and $\alpha,\hbeta,\beta$ are as in
  Definition \ref{definition:pmu}.  A $C^{*}$-pseudo-multiplicative
  unitary $V \colon \Hsource \to \Hrange$ is  {\em regular} if the
  composition
  \begin{align*}
    \lnspan \balpha{1} V \kalpha{2} \rnspan \colon H
    \xright{\kalpha{2}} \Hsource \xright{V}
    \Hrange \xright{\balpha{1}} H
  \end{align*}
  is equal to $\lnspan \alpha \alpha^{*} \rnspan$.
\end{definition}
\begin{remark} \label{remark:pmu-op-regular}
  Evidently, $V$ is regular if and only if $V^{op}$ is regular.
\end{remark}
An example of a regular $C^{*}$-pseudo-multiplicative unitary is given
in Section \ref{section:groupoids}.

We shall show that for a regular $C^{*}$-pseudo-multiplicative
unitary, the associated legs form concrete Hopf $C^{*}$-bimodules. The
first step in this direction is the following:
\begin{proposition} \label{proposition:pmu-regular-algebras}
  Let $V \colon \Hsource \to \Hrange$ be a regular
  $C^{*}$-pseudo-multiplicative unitary. Then
  $\hA(V)$ and $A(V)$ are
  $C^{*}$-algebras.
\end{proposition}
\begin{proof}
  We only prove the assertion
  concerning $\hA=\hA(V)$.  The following diagram commutes and shows
  that $\lnspan \hA \hA^{*}\rnspan = \lnspan \balpha{2} \bbeta{3}
  \Vl{12}^{*}\kbeta{3}\kbeta{2} \rnspan$:
  \begin{align*} \smalldiagram
    \xymatrix@R=20pt@C=35pt{
      & {\sHrange} \ar `r/0pt[rr]^{\kbeta{3}} [rrd] && \\
      H \ar `u/0pt[u] [ur]^(0.45){\kbeta{2}}
      \ar[d]^{\hA^{*}} \ar[r]^(0.45){\kbeta{2}}
      \ar@{}[rd]|{\scriptstyle (D)} &{\sHrange } \ar[r]^(0.4){
        \kalpha{3}} \ar[d]^{V^{*}} & {\Htwo} \ar[d]^{\Vl{12}^{*}}
      \ar[r]^{\Vl{23}} \ar@{}[rddd]|{\scriptstyle (P)} & {\Hthree} \ar `d/0pt[ddd]^{\Vl{12}^{*}} [dddl] \\
      H \ar@{=}[d] \ar@{}[rrd]|{\scriptstyle (R)} &
      {\sHsource} \ar[l]^(0.55){\balpha{2}}
      \ar[r]^(0.4){\kalpha{3}} &{\Hone}
      \ar[d]^{\Vl{23}} &\\H
      \ar[r]^(0.45){\kalpha{2}} \ar[d]^{\hA}
      \ar@{}[rd]|{\scriptstyle (D)} &{\sHsource} \ar[d]^{V}
      &{\Hfive}
      \ar[d]^{\Vl{13}} \ar[l]^(0.6){\balpha{2}}
      &\\H &{\sHrange}
      \ar[l]^(0.55){\bbeta{2}} &{\Hfour} \ar `d[d]
      [dl]^(0.6){\bbeta{3}}
      \ar[l]^(0.6){\balpha{2}} &\\
      & {\sHsource} \ar `l[l]^(0.55){\balpha{2}}
      [lu] & & }
  \end{align*}
  Indeed, the cells labeled by (D) commute by definition, cell (R)
  commutes because $V$ is regular, cell (P) is just Diagram
  \eqref{eq:pmu-pentagon}, and the remaining cells commute because of 
  \eqref{eq:pmu-intertwine}.

  The following commutative diagram shows that $\lnspan \balpha{2}
  \bbeta{3} \Vl{12}^{*}\kbeta{3}\kbeta{2} \rnspan=\hA^{*}$ and
  completes the proof:
  \begin{align*} \smalldiagram \xymatrix@C=35pt@R=20pt{ & {\sHrange}
      \ar `r/0pt[r]^(0.4){\kbeta{3}} [rd] \ar[d]^{ \rho_{\alpha \rt
          \alpha}(B)} &
      \\
      H \ar[r]^(0.45){\kbeta{2}} \ar[d]^{\hA^{*}} \ar `u[u]
      [ur]^(0.45){ \kbeta{2}} & {\sHrange} \ar[d]^{V^{*}} & {\Hthree}
      \ar[l]^(0.6){
        \bbeta{3}} \ar[d]^{\Vl{12}^{*}} \\
      H & {\sHsource} \ar[l]^(0.55){ \balpha{2}} & {\Hfour}
      \ar[l]^(0.6){\bbeta{3}} }
  \end{align*}
  \end{proof}

  \begin{lemma} \label{lemma:pmu-legs-delta} Let $V \colon \Hsource
    \to \Hrange$ be a $C^{*}$-pseudo-multiplicative unitary. Then
    $\hDelta(\hA)$ is equal to the composition
    \begin{align*}
      \smalldiagram \xymatrix{ {\sHsource} \ar[r]^(0.4){\kalpha{3}} &
        {\Hone} \ar[rr]^(0.45){\Vl{13}\Vl{23}} &&
        {\Hfour} \ar[r]^(0.6){\bbeta{3}} & {\sHsource.}}
    \end{align*}
  \end{lemma}
  \begin{proof}
    This follows from the fact that the following diagram commutes:
    \begin{align*} \smalldiagram \xymatrix@C=25pt@R=25pt{
        & \ar@{}[rd]|{\scriptstyle (D)}&& \\
        {\sHsource} \ar `u/0pt[u] `r[urrr]^{ \hDelta(\hA)} [rrr]
        \ar[r]^{V} \ar[d]^{ \kalpha{3}} & \ar@{}[rd]|{\scriptstyle (D)} {\sHrange}
        \ar[r]^{\Id_{\alpha} \tr \hA} \ar[d]^{ \kalpha{3}} &
        {\sHrange} \ar[r]^{V^{*}} &
        {\sHsource} \\
        {\Hone} \ar[r]^{ \Vl{12}} \ar `d[d] `r[drrr]^{\Vl{13}\Vl{23}}
        [rrr] & {\Htwo} \ar[r]^{ \Vl{23}} \ar@{}[rd]|{\scriptstyle (P)} & {\Hthree}
        \ar[u]^{\bbeta{3}} \ar[r]^{\Vl{12}^{*}} & {\Hfour}
        \ar[u]^{\bbeta{3}} \\ &&& }
    \end{align*}
    Again, the cells labeled by (D) commute by definition, cell (P)
    is just Diagram \eqref{eq:pmu-pentagon}, and the other cells
    commute because of \eqref{eq:pmu-intertwine}.
  \end{proof}

The main result of this article is the following:
\begin{theorem} \label{theorem:pmu-regular}
  Let $V \colon \Hsource \to \Hrange$ be a regular
  $C^{*}$-pseudo-multiplicative unitary, where $H$, $\cbasesb$, 
  $\alpha,\hbeta,\beta$ are as in Definition
  \ref{definition:pmu}. Then
  \begin{align*}
    \big(\cbaseosb,H,\hA(V),\hbeta,\alpha,\hDelta\big) \quad \text{and} \quad
    \big(\cbasesb, A(V),\alpha,\beta,\Delta\big)  
  \end{align*}
  are concrete Hopf $C^{*}$-bimodules.
\end{theorem}
\begin{proof}
  First, we show that $\hDelta(\hA) \subseteq \hA
  \fibre{\hbeta}{\frakH}{\alpha} \hA$.
  The following diagram commutes and shows that $\lnspan\hDelta(\hA)
  \kalpha{2} \rnspan = \lnspan \kalpha{2} \hA\rnspan$:
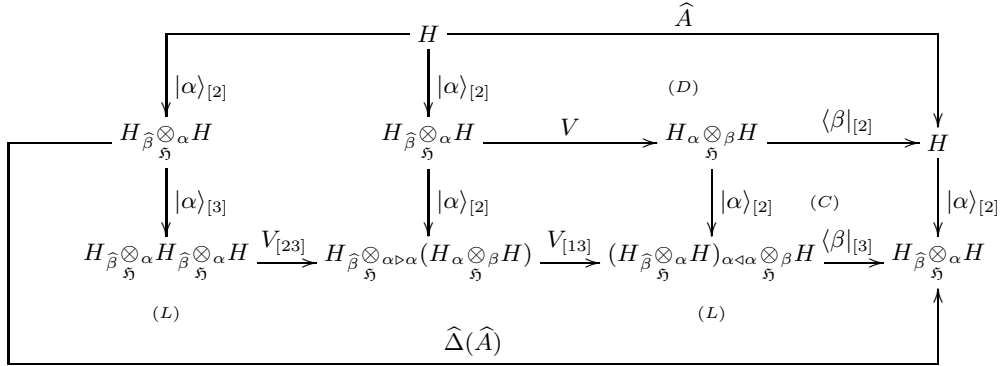
  \begin{align*} \smalldiagram
    \xymatrix@C=22pt@R=25pt{ & & {H} \ar `l/0pt[l]
      [ld]^{\kalpha{2}} \ar@{}[rrd]|{\scriptstyle (D)} \ar `r[rr]^{\hA} [rrd]
      \ar[d]^{\kalpha{2}}&& \\
      & {\sHsource} \ar[d]^{\kalpha{3}} \ar
      `l/0pt[l] `d[ldd] `r[ddrrr]^{\hDelta(\hA)} [drrr] &
      {\sHsource} \ar[d]^{\kalpha{2}}
      \ar[r]^{V} & {\sHrange}
      \ar[r]^(0.6){\bbeta{2}} \ar[d]^{\kalpha{2}}
       \ar@{}[rd]|{\scriptstyle (C)} & 
      H  \ar[d]^{\kalpha{2}}\\
      & {\Hone} \ar[r]^(0.45){
        \Vl{23}}  \ar@{}[d]|{\scriptstyle (L)} & {\Hfive} \ar[r]^{\Vl{13}} & 
      {\Hfour} \ar@{}[d]|{\scriptstyle (L)}
      \ar[r]^(0.6){\bbeta{3}} & {\sHsource}
      \\
      & &&& }
  \end{align*}
  Here, cell (D) commutes by definition, cell (C) commutes by a
  similar calculation as in \eqref{eq:pmu-diagram-commutes}, cell (L)
  commutes by Lemma \ref{lemma:pmu-legs-delta}, an the remaining cells
  commute because of \eqref{eq:pmu-intertwine}.  A similar commutative
  diagram shows that $\lnspan \hDelta(\hA)^{*} \khbeta{1}\rnspan =
  \lnspan \khbeta{1}\hA^{*}\rnspan$.  Thus, $\hDelta(\hA) \subseteq
  \hA \fibre{\hbeta}{\frakH}{\alpha} \hA$. 

  Proposition \ref{proposition:pmu-delta} and Remark
  \ref{remark:fp-morphism-nondegenerate} imply that $\hA
  \fibre{\hbeta}{\frakH}{\alpha} \hA \subseteq {\cal L}(\Hsource)$ is
  nondegenerate. Moreover, by   Proposition
  \ref{proposition:pmu-delta},
  \begin{align*}
    \hDelta \in \Mor\big(\hA_{\alpha}, (\hA
    \fibre{\hbeta}{\frakH}{\alpha} \hA)_{\alpha \lt \alpha}\big) \cap
    \Mor\big(\hA_{\hbeta}, (\hA \fibre{\hbeta}{\frakH}{\alpha}
    \hA)_{\alpha \lt \hbeta}\big).
  \end{align*}

  The following commutative diagram shows that $(\hDelta \ast
  \Id)(\hDelta(\ha))=(\Id \ast\hDelta)(\hDelta(\ha))$ for each $\ha
  \in \hA$ and thus completes the proof:
  \begin{align*} \hspace{-1ex}  \smalldiagram
    
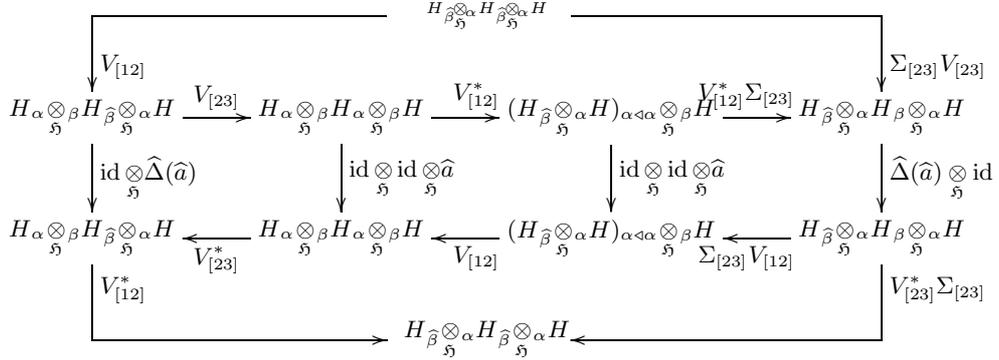
\xymatrix@C=25pt@R=25pt{
      &&& \\
      {\Htwo} \ar[r]^{V_{\leg{23}}} \ar[d]^{\Id \rtensorh
        \hDelta(\ha)} \ar@{} `u/0pt[u] `[urrr]|{\spacestyle \Hone} [rrr]
      \ar@{<-} `u/0pt[u]_{V_{\leg{12}}} [ur]+<17pt,0pt>
      & {\Hthree}
      \ar[r]^{V_{\leg{12}}^{*}} \ar[d]^{\Id \rtensorh \Id \rtensorh \ha} & {\Hfour} \ar[r]^{V_{\leg{12}}^{*}\Sigma_{\leg{23}}} \ar[d]^{\Id \rtensorh
        \Id \rtensorh \ha}&
      {\Hfourlt}  \ar[d]^{\hDelta(\ha) \rtensorh \Id}
      \ar@{<-} `u/0pt[u]_{\Sigma_{\leg{23}}V_{\leg{23}}} [ul]-<15pt,0pt>
      \\
      {\Htwo} \ar@{} `d/0pt[d]
      `r[drrr]|{\textstyle \Hone} [rrr]
      \ar `d/0pt[d]^{V_{\leg{12}}^{*}} `r[dr]+<17pt,0pt>
      & {\Hthree} \ar[l]^{V_{\leg{23}}^{*}} & {\Hfour}
      \ar[l]^{V_{\leg{12}}} &
      {\Hfourlt}  \ar[l]^{\Sigma_{\leg{23}}V_{\leg{12}}}
      \ar `d/0pt[d]^{V_{\leg{23}}^{*}\Sigma_{\leg{23}}} [dl]-<15pt,0pt>
      \\
      &&&}
  \end{align*}
  Here, the upper and lower cells commute by \eqref{eq:pmu-pentagon}, and
  the other cells commute by definition of $\hDelta$ or trivially.
  Hence,
  \begin{align*}
  (\hDelta \ast \Id)\big(\hDelta(\ha)\big) &=
  V_{\leg{12}}^{*}\big(\Id \rtensorh \hDelta(\ha)\big)V_{\leg{12}} \\ &=
  V_{\leg{23}}^{*}\Sigma_{\leg{23}}\big(\hDelta(\ha)\rtensorh
  \Id\big)\Sigma_{\leg{23}}V_{\leg{23}} = (\Id \vnast
  \hDelta)\big(\hDelta(\ha)\big).   \qedhere
  \end{align*}
\end{proof}

For completeness, we include the following additional result:
\begin{proposition}
  Let $V \colon \Hsource \to \Hrange$ be a
  $C^{*}$-pseudo-multiplicative unitary. Then the space
  \begin{align*}
    C:=\lnspan
  \balpha{1}V\kalpha{2}\rnspan \subseteq {\cal L}(H)
  \end{align*}
 is an algebra.  If $V$ is regular, then $C$ is a $C^{*}$-algebra.
\end{proposition} 
\begin{proof}
  The first assertion follows by combining the following commutative
  diagrams:
 \begin{align*} \hspace{-0.25cm} \smalldiagram
    \xymatrix@C=-3pt@R=25pt{ & H \ar[d]^{ \kalpha{2}} \ar[rr]^{C} \ar
      `l/0pt[l] [ld]^{\kalpha{2}} & & H \ar[rd]^(0.65){\kalpha{2}}
      \ar[rr]^{C} & & H  &  \\
      {\sHsource} \ar `d[dd]^(0.75){\kalpha{2}} [ddr] & {\sHsource}
      \ar[r]^(0.55){V} \ar[d]^{\kalpha{3}} & {\sHrange}
      \ar[ru]^(0.35){\balpha{1}} \ar[rd]^{ \kalpha{3}} & & {\sHsource}
      \ar[r]^(0.45){V} & {\sHrange} \ar[u]^{\balpha{1}} & {\sHrange}
      \ar `u[u]^{\balpha{1}} [ul] \\
      & {\Hone} \ar[rr]^{ \Vl{12}} \ar[d]^{\Vl{23}} && {\Htwo}
      \ar[ru]^{\balpha{1}} \ar[rr]^{\Vl{23}} && {\Hthree}
      \ar[u]^{\balpha{1}}   & \\
      & {\Hfive} \ar[rrrr]^{\Vl{13}} & &&& {\Hfour} \ar[u]^{ \Vl{12}}
      \ar `r/0pt[r] [ruu]^(0.25){\balpha{2}} & } 
  \end{align*}
  \begin{gather*} \smalldiagram \xymatrix@C=50pt@R=20pt{ & H \ar
      `l/0pt[l] [ld]^{{\kalpha{2}}} \ar[r]^{C}
      \ar[d]^{{\kalpha{2}}} & H \\
      {\sHsource} \ar `d/0pt[d]^{\kalpha{2}} [dr]
      \ar[r]^(0.55){\rho_{\hbeta \rt \beta}(B)} & {\sHsource} \ar[r]^{
        V} & {\sHrange} \ar[u]^{\balpha{1}} \\
      & {\Hfive} \ar[u]^{\balpha{2}} \ar[r]^{\Vl{13}} & {\Hfour}
      \ar[u]^{\balpha{2}} }
  \end{gather*}  
  To prove the second assertion, we combine the  following two
 commutative diagrams:
  \begin{align*} \smalldiagram \xymatrix@R=20pt@C=35pt{ H
      \ar[d]^{C^{*}} \ar[r]^(0.45){\kalpha{1}} &{\sHrange }
      \ar[r]^(0.4){\kalpha{3}} \ar[d]^{V^{*}} & {\Htwo} \ar[d]^{
        \Vl{12}^{*}} \ar `r/0pt[r] [rddd]^{\Vl{23}}   & \\
      H \ar@{=}[d] &{\sHsource} \ar[l]^(0.55){\balpha{2}}
      \ar[r]^(0.4){ \kalpha{3}} &{\Hone} \ar[d]^{\Vl{23}} & \\H
      \ar[r]^(0.45){\kalpha{2}} \ar[d]^{C} &{\sHsource} \ar[d]^{ V}
      &{\Hfive} \ar[d]^{\Vl{13}} \ar[l]^(0.6){ \balpha{2}} &\\H &
      {\sHrange} \ar[l]^(0.55){\balpha{1}} & {\Hfour}
      \ar[l]^(0.6){\balpha{2}} &{\Hthree} \ar[l]^(0.4){\Vl{12}^{*}}
      \ar `d[d]
      [dll]^(0.6){\balpha{1}} \\
      & {\sHrange} \ar `l[l]^(0.55){\balpha{1}} [lu] & & }
  \end{align*}
  \begin{align*} \smalldiagram \xymatrix@C=35pt@R=25pt{ & {\sHrange}
      \ar `r/0pt[r]^(0.4){ \kbeta{3}} [rd] \ar[d]^{\rho_{\alpha \rt
          \alpha}(B)} &
      \\
      H \ar[r]^(0.45){\kbeta{2}} \ar[d]^{C} \ar `u[u]
      [ur]^(0.45){\kbeta{2}} & {\sHrange} \ar[d]^{V^{*}} & {\Hthree} \ar[l]^(0.6){
        \bbeta{3}} \ar[d]^{\Vl{12}^{*}} \\
      H & {\sHsource} \ar[l]^(0.55){ \balpha{2}} & {\Hfour}
      \ar[l]^(0.6){\bbeta{3}} }
  \end{align*}
\end{proof}

\section{Locally compact groupoids}
\label{section:groupoids}

 The prototypical example of a $C^{*}$-pseudo-multiplicative unitary is the
 unitary associated to a locally compact groupoid. The underlying
 pseudo-multiplicative unitary was introduced by Vallin
 \cite{vallin:2}, and associated unitaries on $C^{*}$-modules were
 discussed  in \cite{timmermann:thesis,timmermann}.  
We construct the $C^{*}$-pseudo-multiplicative unitary, prove that it
is regular, and show that the associated legs are just the function
algebra of the groupoid on one side and the reduced groupoid
$C^{*}$-algebra on the other side.
For background on groupoids, Haar measures, and quasi-invariant
measures, see \cite{renault} or \cite{paterson}.

Let $G$ be a locally compact, Hausdorff, second countable groupoid. We
denote its unit space by $G^{0}$, its range map by $r_{G}$, its source
map by $s_{G}$, and put $G^{u}:=r_{G}^{-1}(\{u\})$,
$G_{u}:=s_{G}^{-1}(u)$ for each $u \in G^{0}$.

We assume that $G$ has a left Haar system $\lambda$, and denote the
associated right Haar system by $\lambda^{-1}$. Let $\mu$ be a measure
on $G^{0}$ and denote by $\nu$ the measure on $G$ given by
 \begin{align*}
   \int_{G} f \,d \nu &:= \int_{G^{0}} \int_{G^{u}} f(x) \, d\lambda^{u}(x)
\,   d\mu(u) \quad \text{for all } f \in C_{c}(G).
 \end{align*}
 The push-forward of $\nu$ via the inversion map $G \to G$, $x \mapsto
 x^{-1}$, is denoted by $\nu^{-1}$; evidently,
 \begin{align*}
   \int_{G} f d\nu^{-1} &= \int_{G^{0}}  \int_{G_{u}} f(x)
   d\lambda^{-1}_{u}(x) \, d\mu(u).
 \end{align*}
We assume that
 the measure $\mu$ is  quasi-invariant, i.e., that $\nu$ and
 $\nu^{-1}$ are equivalent. Note that there always exist sufficiently many
 quasi-invariant measures \cite{renault}.  We denote  by
 $D:=d\nu/d\nu^{-1}$ the Radon-Nikodym derivative.

 The measure $\mu$ defines a tracial proper weight on the
 $C^{*}$-algebra $C_{0}(G^{0})$, which we denote by $\mu$
 again. Moreover, we denote by $\cbasesb$ the $C^{*}$-base associated
 to $\mu$ as in Example \ref{example:rtp-base-weight}; thus,
 $\frakH=L^{2}(G^{0},\mu)$. Note that $\cbasesb=\cbaseosb$ because
 $C_{0}(G^{0})$ is commutative.

 Put $H:=L^{2}(G,\nu)$ and define representations $r,s\colon
 C_{0}(G^{0}) \to {\cal L}(L^{2}(G,\nu)$ by
 \begin{align*}
  \big(r(f)\xi\big)(x) &:=
 f\big(r_{G}(x)\big)\xi(x), & 
\big(s(f)\xi\big)(x) &:=
 f\big(s_{G}(x)\big) \xi(x)
 \end{align*}
 for all $x \in G$, $\xi \in C_{c}(G)$, and $f \in C_{0}(G^{0})$.

 The space $C_{c}(G)$ forms a pre-$C^{*}$-module over $C_{0}(G^{0})$
 with respect to the structure maps
 \begin{align*}
   \begin{aligned}
     \langle \xi'|\xi\rangle(u)&= \int_{G^{u}}
     \overline{\xi'(x)}\xi(x) d\lambda^{u}(x), & (\xi f)(x) &=
     \xi(x)f(r_{G}(x)),
   \end{aligned}
 \end{align*}
and also with respect to the structure maps
\begin{align*}
  \begin{aligned}
    \langle \xi'|\xi\rangle(u)&= \int_{G_{u}} \overline{\xi'(x)}\xi(x)
    d\lambda^{-1}_{u}(x), & (\xi f)(x) &= \xi(x)f(s_{G}(x)).
  \end{aligned}
\end{align*}
We denote the completions of these  pre-$C^{*}$-modules by
$L^{2}(G,\lambda)$ and $L^{2}(G,\lambda^{-1})$, respectively.
\begin{proposition} \label{proposition:groupoid-factorizations}
  There exist isometric embeddings
  \begin{align*}
    j\colon L^{2}(G,\lambda) &\to {\cal L}(\frakH,H) &&\text{and}& \hat j \colon
    L^{2}(G,\lambda^{-1}) &\to {\cal L}\big(\frakH,H\big)
  \end{align*}
  such that for all $\xi \in C_{c}(G)$,  $\zeta \in L^{2}(G^{0},\mu)$, $x \in G$,
  \begin{align*}
    \big(j(\xi) \zeta\big)(x) &= \xi(x)\zeta(r_{G}(x)), &
    \big(\hat j(\xi)\zeta\big)(x) &= \xi(x) D^{-1/2}(x) \zeta(s_{G}(x)).
  \end{align*}
  The images $\alpha:=j(L^{2}(G,\lambda))$ and $\hbeta := \hat
  j(L^{2}(G,\lambda^{-1}))$ are compatible $C^{*}$-factorizations of
  $H$ with respect to $\cbasesb$. We have $\rho_{\alpha}=r$ and
  $\rho_{\hbeta}=s$.  Finally, $j$ and $\hat j$ are unitary maps of
  $C^{*}$-modules over $C_{0}(G^{0}) \cong \frakB$.
\end{proposition}
\begin{proof}
  Let $\xi,\xi ' \in C_{c}(G)$ and $\zeta,\zeta' \in C_{c}(G^{0})$. Then
  \begin{align*}
    \big\langle j(\xi')\zeta' \big| j(\xi)\zeta\big\rangle_{H} &=
    \int_{G^{0}} \int_{G^{u}} \overline{\xi'(x)\zeta'(r_{G}(x))}
    \xi(x) \zeta(r_{G}(x)) d\lambda^{u}(x) d\mu(u)
    \\
    &= \int_{G^{0}} \langle\xi'|\xi\rangle_{L^{2}(G,\lambda)}
    \overline{\zeta'(u)} \zeta(u) d\mu(u) = \big\langle \zeta'\big|
    \pi_{\mu}\big(\langle\xi'|\xi\rangle_{L^{2}(G,\lambda)}\big)
    \zeta\big\rangle_{\frakH} \shortintertext{and} \big\langle \hat
    j(\xi')\zeta'\big|\hat j(\xi)\zeta \big\rangle_{H} &= \int_{G}
    \overline{\xi'(x)\zeta'(s_{G}(x))}
    \xi(x)\zeta(s_{G}(x))  D^{-1}(x)d\nu(x) \\
    &= \int_{G} \overline{\xi'(x)\zeta'(s_{G}(x))}
    \xi(x)\zeta(s_{G}(x)) d\nu^{-1}(x) \\
    &= \int_{G^{0}} \int_{G_{u}} \overline{\xi'(x)}\xi(x)
    \overline{\zeta'(u)}\zeta(u) d\lambda^{-1}_{u}(x) d\mu(u) 
    = \big \langle \zeta' \big|
    \pi_{\mu}\big(\langle\xi'|\xi\rangle_{L^{2}(G,\lambda^{-1})}\big)
    \zeta\big\rangle_{\frakH}.
  \end{align*}
  These calculations prove the existence of the isometric embeddings
  $j$ and $\hat j$. Straightforward arguments show the images of these
  embeddings are $C^{*}$-factorizations, and the calculations above
  show that these embeddings are unitary maps of $C^{*}$-modules.  The
  assertions $\rho_{\alpha}=r$, $\rho_{\hbeta}=s$, and $\alpha\perp
  \hbeta$ follow from routine calculations and arguments.
\end{proof}
Put $\beta:=\alpha$.

Denote by $\tilde \mu$ the extension of the weight $\mu$ to the von
Neumann algebra $\pi_{\mu}(C_{0}(G^{0}))'' \cong
L^{\infty}(G^{0},\mu)$, and by $\tilde r$ and $\tilde s$ the
extensions of $r=\rho_{\alpha}$ and $s=\rho_{\hbeta}$ to
$L^{\infty}(G^{0},\mu)$. From now on, we identify
\begin{align} \label{eq:groupoid-iso1}
  H \vtensor{\tilde s}{\tilde \mu}{\tilde r} H &\cong \Hsource, &
  H \vtensor{\tilde r}{\tilde \mu}{\tilde r} H &\cong \Hrange
\end{align}
as in Proposition \ref{proposition:rtp-comparison}.  By
\cite{vallin:2}, the Hilbert spaces above can be described as follows.
Define a measure $\nu^{2}_{s,r}$ on $G^{2}_{s,r}:=\{(x,y)\in G \times
G \mid s(x)=r(y)\}$ by
 \begin{gather*}
   \int_{G^{2}_{s,r}} \!\! f\, d\nu^{2}_{s,r} := \int_{G^{0}}
   \int_{G^{u}} \int_{G^{s_{G}(x)}} f(x,y) \, d\lambda^{s_{G}(x)}(y)
   \, d\lambda^{u}(x) \, d\mu(u),
 \end{gather*}
 and a measure $\nu^{2}_{r,r}$ on $G^{2}_{r,r}:=\{ (x,y) \in
 G^{2}\mid r_{G}(x)=r_{G}(y)\}$  by
 \begin{gather*}
   \int_{G^{2}_{r,r}} \!\! g\, d\nu^{2}_{r,r} := \int_{G^{0}}
   \int_{G^{u}}\int_{G^{u}} g(x,y)\, d\lambda^{u}(y)\, d\lambda^{u}(x)\,
   d\mu(u),
 \end{gather*}
where $f \in C_{c}(G^{2}_{s,r})$ and $g\in C_{c}(G^{2}_{r,r})$. Then 
\begin{align} \label{eq:groupoid-iso2}
    H \vtensor{\tilde s}{\tilde \mu}{\tilde r} H &\cong
    L^{2}\big(G^{2}_{s,r},\nu^{2}_{s,r}\big), &
    H \vtensor{\tilde r}{\tilde \mu}{\tilde r} H &\cong    
    L^{2}\big(G^{2}_{r,r},\nu^{2}_{r,r}\big).
\end{align}
By \cite{vallin:2}, there exists a pseudo-multiplicative unitary
\begin{align*}
  V \colon \Hsource \to \Hrange
\end{align*}
such that, with respect to the isomorphisms \eqref{eq:groupoid-iso1}
and \eqref{eq:groupoid-iso2},
\begin{align*}
  (V\zeta)(x,y) &= \zeta(x,x^{-1}y) \quad \text{for all } \zeta \in
     L^{2}(G^{2}_{s,r},\, \nu^{2}_{s,r}), \, (x,y) \in G^{2}_{r,r}.
\end{align*}
Using the identifications \eqref{eq:groupoid-iso1}, we consider $V$ as
a unitary $\Hsource \to \Hrange$.
\begin{theorem}
  The unitary $V$ is a $C^{*}$-pseudo-multiplicative unitary.
\end{theorem}
\begin{proof}
  We only need to prove 
  \begin{gather} \label{eq:groupoid-intertwine}
    \begin{aligned}
      V_{*}(\alpha \lt \alpha) &= \alpha \rt \alpha, &
      V_{*}(\hbeta \rt \beta) &= \hbeta \lt \beta, \\
      V_{*}(\hbeta \rt \hbeta) &= \alpha \rt \hbeta, & V_{*}(\beta \lt
      \alpha) &= \beta \lt \beta.
    \end{aligned}
    \end{gather}
    By Lemma \ref{lemma:itp-concrete} and Proposition
    \ref{proposition:groupoid-factorizations}, we can identify the
    factorizations $\alpha \lt \alpha = \beta \lt \alpha$, $\hbeta \rt
    \beta$, $\hbeta \rt \hbeta$ of $\Hsource$ as $C^{*}$-modules with
    the internal tensor products
  \begin{align} \label{eq:groupoid-itp1}
    L^{2}(G,\lambda) &{_{s}\tl} L^{2}(G,\lambda), &
    L^{2}(G,\lambda^{-1}) &\tr_{r} L^{2}(G,\lambda), &
    L^{2}(G,\lambda^{-1}) &\tr_{r} L^{2}(G,\lambda^{-1}),
  \end{align}
  and the factorizations $\alpha \rt \alpha$, $\hbeta \lt \beta$,
  $\alpha \rt \hbeta$, $\beta \lt \beta$ of $\Hrange$ as
  $C^{*}$-modules with the internal tensor products
  \begin{align} \label{eq:groupoid-itp2}
    \begin{aligned}
      L^{2}(G,\lambda) &\tr_{r} L^{2}(G,\lambda), &
      L^{2}(G,\lambda^{-1}) &{_{r}\tl} L^{2}(G,\lambda), \\
      L^{2}(G,\lambda) &\tr_{r} L^{2}(G,\lambda^{-1}), &
      L^{2}(G,\lambda) &{_{r}\tl} L^{2}(G,\lambda).
    \end{aligned}
  \end{align}
  Note that $\alpha \rt \alpha = \beta \rt \beta$.  The internal
  tensor products in \eqref{eq:groupoid-itp1} can be identified with
  certain completions of $C_{c}(G^{2}_{s,r})$, and the internal tensor
  products in \eqref{eq:groupoid-itp2} can be identified with certain
  completions of $C_{c}(G^{2}_{r,r})$.  In each of the equations in
  \eqref{eq:groupoid-intertwine}, the unitary $V$ maps the subspace
  $C_{c}(G^{2}_{s,r})$ of the $C^{*}$-factorization on the left hand
  side to the subspace $C_{c}(G^{2}_{r,r})$ of the
  $C^{*}$-factorization of the right hand side because $V$ is the
  transpose of a homeomorphism $G^{2}_{r,r} \to G^{2}_{s,r}$.
  The claim follows.
\end{proof}

\begin{proposition}
  The unitary $V$ is regular.
\end{proposition}
\begin{proof}
  For each $\xi,\xi' \in C_{C}(G)$, $\zeta \in L^{2}(G,\nu)$, and $y
  \in G$,
  \begin{align*}
    \big(\langle j(\xi')|_{\leg{1}} V |j(\xi)\rangle_{\leg{2}}
    \zeta\big)(y) &= \int_{G^{r_{G}(y)}} \overline{\xi'(x)}
    \zeta(x)\xi(x^{-1}y)
    d\lambda^{r_{G}(y)}(x),\\
    \big(j(\xi')j(\xi)^{*}\zeta\big)(y) &= \xi'(y) \int_{G^{r_{G}(y)}}
    \overline{\xi(x)} \zeta(x)d\lambda^{r_{G}(y)}(x).
  \end{align*}
  These equations and standard arguments show that $\lnspan
  \balpha{1}V \kalpha{2}\rnspan$ and $\lnspan \alpha\alpha^{*}\rnspan$
  coincide with the closed span of operators on
  $L^{2}(G,\nu)$ of the form
  \begin{align*}
    \zeta \mapsto \Big(y \mapsto \int_{G^{r_{G}(y)}} f(x,y)
    \zeta(x)d\lambda^{r_{G}(y)}(x)\Big),
  \end{align*}
  where $f \in C_{c}(G^{2}_{r,r})$.
\end{proof}

By Theorem \ref{theorem:pmu-regular}, the regular
$C^{*}$-pseudo-multiplicative  unitary $V$  gives rise to two concrete
Hopf $C^{*}$-bimodules
\begin{align*}
  \big(\cbaseosb,H,\hA(V),\hbeta,\alpha,\hDelta\big) \quad
  \text{and} \quad \big(\cbasesb, A(V),\alpha,\beta,\Delta\big).
\end{align*}
Denote by $m \colon C_{0}(G) \to {\cal L}(L^{2}(G,\nu))$ the
representation given by multiplication operators. Recall that for each
$g \in C_{c}(G)$, there exists a unique operator $L(g) \in {\cal
  L}(L^{2}(G,\nu))$ such that 
\begin{align*}
\big(L(g)\zeta\big)(y) = \int_{G^{r_{G}(y)}}
g(x)D^{-1/2}(x) \zeta(x^{-1}y) d\lambda^{r_{G}(y)}(x) \quad \text{for
  all } y \in G,\, \zeta \in L^{2}(G,\nu),
\end{align*}
 and that the  reduced groupoid $C^{*}$-algebra $C^{*}_{r}(G)$ is the
 closed linear span of all operators of the $L(g)$, where $g \in C_{c}(G)$
\cite{renault}.
\begin{theorem}
\begin{enumerate}
\item $\hA(V)=m(C_{0}(G)) \cong C_{0}(G)$, and for each $f \in
  C_{0}(G)$, $\zeta \in \Hsource \cong
  L^{2}(G^{2}_{s,r},\nu^{2}_{s,r})$, $(x,y) \in G^{2}_{s,r}$,
    \begin{align*}
      \big(\hDelta(m(f))\zeta\big)(x,y) = f(xy) \zeta(x,y).
    \end{align*}
  \item   $A(V)=C^{*}_{r}(G)$, and for each $g \in
    C_{c}(G)$, $\zeta \in \Hrange \cong
    L^{2}(G^{2}_{r,r},\nu^{2}_{r,r})$, $(x,y) \in G^{2}_{r,r}$,
    \begin{align*}
      \big(\Delta(L(g))\zeta\big)(x,y) = \int_{G^{r_{G}(x)}}
      g(z)D^{-1/2}(z) \zeta(z^{-1}x,z^{-1}y) d\lambda^{r_{G}(x)}(z).
    \end{align*}
  \end{enumerate}
\end{theorem}
\begin{proof}
  i) Let $\xi,\xi' \in C_{c}(G)$ and put $\ha_{\xi,\xi'}:=\langle
  j(\xi)|_{\leg{2}}V|j(\xi')\rangle_{\leg{2}}$. For all $\zeta,\zeta'
  \in L^{2}(G,\nu)$,
  \begin{align*}
    \langle \zeta| \ha_{\xi,\xi'} \zeta'\rangle &=
    \langle \zeta \tl j(\xi)|V(\zeta' \tl j(\xi'))\rangle \\ &=
    \int_{G} \int_{G^{r_{G}(x)}} \overline{\zeta(x)\xi(y)}
    \zeta'(x)\xi'(x^{-1}y) d\lambda^{r_{G}(x)}(y) d\nu(x) 
    = \langle \zeta|m(f) \zeta'\rangle,
  \end{align*}
  where $f \in C_{c}(G)$ is given by
  \begin{align*}
    x \mapsto \int_{G^{r_{G}(x)}} \overline{\xi(y)}\xi'(x^{-1}y)
    d\lambda^{r_{G}(x)}(y).
  \end{align*}
  Standard arguments show that the closed linear span of all operators
  $m(f)$ with $f \in C_{c}(G)$ as above is equal to
  $m(C_{0}(G))$. Hence, $\hA(V)=m(C_{0}(G))$. 

  For each $f \in C_{0}(G)$, $\zeta \in
  L^{2}(G^{2}_{s,r},\nu^{2}_{s,r})$, and $(x,y) \in G^{2}_{s,r}$,
  \begin{align*}
    \big( \hDelta(m(f))\zeta\big)(x,y) &= \big(V^{*}( \Id \rtensorh
    m(f))V\zeta \big) (x,y) \\
    &= \big(( \Id \rtensorh m(f))V\zeta \big) (x,xy) =
    f(xy)(V\zeta)(x,xy) = f(xy)\zeta(x,y).
  \end{align*}

  ii) Let $\xi,\xi' \in C_{c}(G)$ and put $a_{\xi,\xi'}:=\langle
  j(\xi)|_{\leg{1}} V |\hat j(\xi')\rangle_{\leg{1}}$.  For all
  $\zeta,\zeta' \in L^{2}(G,\nu)$,
  \begin{align*}
    \langle \zeta|a_{\xi,\xi'}\zeta'\rangle &=
    \langle  j(\xi) \tr \zeta| V(\hat j(\xi') \tr \zeta')\rangle \\
    &= \int_{G} \int_{G^{r_{G}(y)}} \overline{\xi(x)\zeta(y)}
    \xi'(x)D^{-1/2}(x) \zeta'(x^{-1}y) d\lambda^{r_{G}(y)}(x) d\nu(y) 
    =    \langle \zeta|L(g)\zeta'\rangle,
  \end{align*}
  where $g \in C_{c}(G)$ is given by
  \begin{align*}
    x \mapsto \int_{G^{r_{G}(x)}} \overline{\xi(x)} \xi'(x).
  \end{align*}
  By definition, the closed linear span of all operators $L(g)$ with
  $g \in C_{c}(G)$ is equal to $C^{*}_{r}(G)$, so $A(V)=C^{*}_{r}(G)$.

  Finally, for each $g \in C_{c}(G)$, $\zeta \in
  L^{2}(G^{2}_{r,r},\nu^{2}_{r,r})$, and $(x,y) \in G^{2}_{r,r}$,
  \begin{align*}
    \big( \Delta(L(g))\zeta\big)(x,y) &= \big(V(L(g) \rtensorh
   \Id)V^{*}\zeta \big) (x,y) \\
    &= \big((L(g)  \rtensorh \Id)V^{*}\zeta \big) (x,x^{-1}y) \\
    &=
    \int_{G^{r_{G}(x)}} g(z) D^{-1/2}(z) (V^{*}\zeta)(z^{-1}x,x^{-1}y)
    d\lambda^{r_{G}(x)}(z)   \\
    &=    \int_{G^{r_{G}(x)}} g(z) D^{-1/2}(z) \zeta(z^{-1}x,z^{-1}y)
    d\lambda^{r_{G}(x)}(z). \qedhere
  \end{align*}
\end{proof}

\def\cprime{$'$}


\begin{thebibliography}{10}

\bibitem{baaj:2}
S.~Baaj and G.~Skandalis.
\newblock Unitaires multiplicatifs et dualit\'e pour les produits crois\'es de
  {$C\sp *$}-alg\`ebres.
\newblock {\em Ann. Sci. \'Ecole Norm. Sup. (4)} \textbf{26} (1993), 425--488.

\bibitem{blanchard}
{\'E}.~Blanchard.
\newblock D\'eformations de {$C\sp *$}-alg\`ebres de {H}opf.
\newblock {\em Bull. Soc. Math. France} (1) \textbf{124} (1996), 141--215.

\bibitem{echterhoff}
S.~Echterhoff, S.~Kaliszewski, J.~Quigg, and I.~Raeburn.
\newblock A categorical approach to imprimitivity theorems for {$C\sp
  *$}-dynamical systems.
\newblock {\em Mem. Amer. Math. Soc.} (850) \textbf{150} (2006).

\bibitem{enock:10}
M.~Enock.
\newblock Quantum groupoids of compact type.
\newblock {\em J. Inst. Math. Jussieu} (1) \textbf{4} (2005), 29--133.

\bibitem{enock:1}
M.~Enock and J.-M. Vallin.
\newblock Inclusions of von {N}eumann algebras, and quantum groupoids.
\newblock {\em J. Funct. Anal.} (2) \textbf{172} (2000), 249--300.

\bibitem{enock:2}
M.~Enock and J.-M. Vallin.
\newblock Inclusions of von {N}eumann algebras and quantum groupoids. {II}.
\newblock {\em J. Funct. Anal.} (1) \textbf{178} (2000), 156--225.

\bibitem{vaes:1}
J.~Kustermans and S.~Vaes.
\newblock Locally compact quantum groups.
\newblock {\em Ann. Sci. \'Ecole Norm. Sup.} (6) \textbf{33} (2000), 837--934.

\bibitem{vaes:30}
J.~Kustermans and S.~Vaes.
\newblock Locally compact quantum groups in the von {N}eumann algebraic
  setting.
\newblock {\em Math. Scand.} (1) \textbf{92} (2003), 68--92.

\bibitem{lance}
E.~C. Lance.
\newblock {\em Hilbert {$C\sp *$}-modules. A toolkit for operator algebraists},
  volume 210 of {\em London Mathematical Society Lecture Note Series}.
\newblock Cambridge University Press, Cambridge, 1995.

\bibitem{lesieur}
F.~Lesieur.
\newblock {\em Groupo\"ides quantiques mesur\'es: axiomatique, \'etude,
  dualit\'e, exemples.}
\newblock PhD thesis, Universite d'Orl\'eans, 2003.

\bibitem{masuda}
T.~Masuda, Y.~Nakagami, and S.~L. Woronowicz.
\newblock A {$C\sp \ast$}-algebraic framework for quantum groups.
\newblock {\em Internat. J. Math.} (9) \textbf{14} (2003), 903--1001.

\bibitem{ouchi} M.~O'uchi.  \newblock Pseudo-multiplicative unitaries
  on {H}ilbert {$C\sp \ast$}-modules.  \newblock {\em Far East J.
    Math. Sci. (FJMS)}, Special Volume, Part II (2001), 229--249.

\bibitem{ouchi:2}
M.~O'uchi.
\newblock Pseudo-multiplicative unitaries associated with inclusions of
  finite-dimensional {$C\sp *$}-algebras.
\newblock {\em Linear Algebra Appl.} \textbf{341} (2002), 201--218.


\bibitem{ouchi:3} M.~O'uchi.  \newblock Coring structures associated
  with multiplicative unitary operators on {H}ilbert {$C\sp
    *$}-modules.  \newblock {\em Far East J. Math. Sci. (FJMS)} (2)
  \textbf{11} (2003), 121--136.

\bibitem{paterson}
A.~L.~T. Paterson.
\newblock {\em Groupoids, inverse semigroups, and their operator algebras},
  volume 170 of {\em Progress in Mathematics}.
\newblock Birkh\"auser Boston Inc., Boston, MA, 1999.

\bibitem{renault} J.~Renault.  \newblock {\em A groupoid approach to
    {$C\sp{\ast} $}-algebras}, volume 793 of {\em Lecture Notes in
    Mathematics}.  \newblock Springer, Berlin, 1980.

\bibitem{timmermann} T.~Timmermann. \newblock{\em
    Pseudo-multiplicative unitaries on $C^{*}$-modules and {H}opf
    $C^{*}$-families I}. to appear in \newblock{\em J.\ of
    Noncommutative Geometry}.

\bibitem{timmermann:thesis}
T.~Timmermann.
\newblock {\em Pseudo-multiplicative unitaries and pseudo-{K}ac systems on
  ${C}^*$-modules}.
\newblock PhD thesis, Universit\"at M\"unster, 2005.
\newblock Available at
  {\texttt{www.math.uni-muenster.de/sfb/about/publ/heft394.ps}}.

\bibitem{vallin:1}
J.-M. Vallin.
\newblock Bimodules de {H}opf et poids op\'eratoriels de {H}aar.
\newblock {\em J. Operator Theory}, 35(1):39--65, 1996.

\bibitem{vallin:2}
J.-M. Vallin.
\newblock Unitaire pseudo-multiplicatif associ\'e \`a un groupo\"\i de.
  {A}pplications \`a la moyennabilit\'e.
\newblock {\em J. Operator Theory} (2) \textbf{44} (2000), 347--368.

\end{thebibliography}
\end{document}